\documentclass[12pt]{amsart}
\usepackage{latexsym}
\usepackage{amsfonts}
\usepackage{amssymb}
\usepackage{xspace,enumerate,color}

\usepackage[all]{xypic}
\definecolor{darkgreen}{rgb}{0,0.45,0} 

\usepackage[pagebackref,colorlinks,citecolor=darkgreen,linkcolor=darkgreen]{hyperref}

\definecolor{lightgrey}{rgb}{0.666666,0.666666,0.666666}

\renewcommand{\epsilon}{\varepsilon}
\renewcommand{\phi}{\varphi}

\newtheorem{theo}{Theorem}[section]
\newtheorem{lemma}[theo]{Lemma}
\newtheorem{propo}[theo]{Proposition}
\newtheorem{coro}[theo]{Corollary}
\theoremstyle{definition}
\newtheorem{defi}[theo]{Definition}
\theoremstyle{remark}
\newtheorem{rem}[theo]{Remark}
\newtheorem{exam}[theo]{Example}

\newcommand\Inj{\operatorname{Inj}}

\newcommand\Set{\operatorname{\bf Set}}
\newcommand\SSet{\ensuremath{\operatorname{\bf SSet}}\xspace}

\newcommand\Lex{\operatorname{\bf Lex}}
\newcommand\Cat{\ensuremath{\operatorname{\bf Cat}}\xspace}
\newcommand\Gpd{\ensuremath{\operatorname{\bf Gpd}}\xspace}
\newcommand\vcat{\ensuremath{\operatorname{\bf \cv-Cat}}\xspace}
\newcommand\Twocat{\ensuremath{\operatorname{\bf 2-Cat}}\xspace}

\DeclareMathOperator\ev{ev}
\DeclareMathOperator\Ho{Ho}
\DeclareMathOperator\ho{ho}
\DeclareMathOperator\Hor{Hor}
\DeclareMathOperator\Int{Int}

\DeclareMathOperator\colim{colim}
\DeclareMathOperator\hocolim{hocolim}

\newcommand\ca{\ensuremath{\mathcal {A}}\xspace}
\newcommand\cb{\ensuremath{\mathcal {B}}\xspace}
\newcommand\cc{\ensuremath{\mathcal {C}}\xspace}
\newcommand\cd{\ensuremath{\mathcal {D}}\xspace}

\newcommand\cf{\ensuremath{\mathcal {F}}\xspace}
\newcommand\cg{\ensuremath{\mathcal {G}}\xspace}
\newcommand\ch{\mathcal {H}}
\newcommand\ci{\ensuremath{\mathcal {I}}\xspace}
\newcommand\cj{\ensuremath{\mathcal {J}}\xspace}
\newcommand\ck{\ensuremath{\mathcal {K}}\xspace}
\newcommand\cl{\ensuremath{\mathcal {L}}\xspace}
\newcommand\cm{\ensuremath{\mathcal {M}}\xspace}
\newcommand\cn{\ensuremath{\mathcal {N}}\xspace}

\newcommand\cw{\ensuremath{\mathcal {W}}\xspace}
\newcommand\cx{\ensuremath{\mathcal {X}}\xspace}

\newcommand\cv{\ensuremath{\mathcal {V}}\xspace}

\newcommand{\ox}{\otimes}

\newcommand{\op}{^{\textnormal{op}}}
\newcommand{\fib}{_{\textnormal{fib}}}
\newcommand{\cof}{_{\textnormal{cof}}}

\renewcommand{\tilde}{\widetilde}

\newcommand{\dhom}{\ensuremath{\mathbb{H}}\xspace}

\date{April 27, 2015}
 
\begin{document}
\title[Homotopy locally presentable enriched categories]
{Homotopy locally presentable enriched categories}

\author{Stephen Lack}
\address{Department of Mathematics, Macquarie University NSW 2109, Australia}
\email{steve.lack@mq.edu.au}

\author{Ji\v r\'i Rosick\'{y}}
\address{Department of Mathematics and Statistics, Masaryk University, Faculty of Sciences, Kotl\'{a}\v{r}sk\'{a} 2, 611 37 Brno, Czech Republic}
\email{rosicky@math.muni.cz}

\thanks{Lack gratefully acknowledges the support of the Australian Research Council Discovery Grant DP130101969 and an ARC Future Fellowship. Rosick\'y gratefully acknowledges the support of MSM 0021622409 and GA\v CR 210/11/0528.}

\begin{abstract}
We develop a homotopy theory of categories enriched in a monoidal model category \cv. In particular, we deal with homotopy weighted limits and colimits, and homotopy local presentability. The main result, which was known for simplicially-enriched categories, links homotopy locally presentable $\cv$-categories with combinatorial model $\cv$-categories, in the case where all objects of \cv are cofibrant.
\end{abstract}
\keywords{monoidal model category, enriched model category, weighted homotopy colimit}

\maketitle
 
\section{Introduction}
There is a fruitful interaction between enriched category theory and homotopy theory, of which the classical case is simplicial homotopy theory: the homotopy theory of simplicial model categories. Moreover, since Dwyer-Kan equivalences provide a suitable notion of weak equivalence between simplicial categories, one can develop a homotopy theory of simplicial categories. In fact there is a model category structure on the category of small simplicial categories, in which the weak equivalences are the Dwyer-Kan equivalences \cite{Be}. This model category is Quillen equivalent to the model category of small quasi-categories \cite{Be1}. Dwyer-Kan equivalences and fibrations also make sense for large simplicial categories, and fibrant simplicial categories correspond to quasi-categories. In particular, the homotopy locally presentable simplicial categories introduced in \cite{R1} correspond to the locally presentable quasi-categories of \cite{L}, in the sense of the following result from \cite{R1}: a  fibrant simplicial category $\ck$ is homotopy locally presentable if and only if it admits a Dwyer-Kan equivalence to the simplicial category $\Int\cm$ of cofibrant and fibrant objects in a combinatorial simplicial model category $\cm$.  There is a gap in the proof in \cite{R1} which we correct by assuming the large cardinal axiom called Vop\v enka's principle (see \cite{AR}). 

Dwyer-Kan equivalences and fibrations can be defined for $\cv$-categories over any monoidal model category $\cv$, and so one can ask whether there is a corresponding model category structure on the category \vcat of all (small) \cv-categories. This question has been studied by various authors under various hypotheses \cite{BM,L,M}; also particular examples have been studied, such as $\cv=\SSet$ \cite{Be}, $\cv=\Cat$ \cite{La2}, and $\cv=\Twocat$ \cite{La3}.

The aim of our paper is to introduce homotopy locally presentable $\cv$-categories, and to give a characterization analogous to that in the case of simplicial categories.
As in the simplicial case, we need Vop\v enka's principle for this.
% {\grey
% Even in the simplicial case, the full theory holds only under the assumption of  the large cardinal axiom called Vop\v enka's principle (see \cite{AR}). }

Just as the definition of (enriched) locally presentable categories \cite{K} involves (weighted) limits and colimits, the definition of {\em homotopy} locally presentable categories involves weighted {\em homotopy} limits and colimits. We define these as weighted limits or colimits whose weight is cofibrant in the projective model structure. This emerges from a classical calculation of homotopy limits and colimits in simplicial model categories; see  \cite{H} for example.

In what follows, $\cv$ will be a monoidal model category in the sense used in \cite{L}; in particular this means that the unit object $I$ is cofibrant, rather than the weaker condition introduced in \cite{Ho}. For such a \cv, there is a notion of model \cv-category, as defined in \cite{Ho}. We also suppose that \cv is cofibrantly generated. We further suppose that \cv is locally presentable as a closed category, in the sense of \cite{K}. For such a \cv, there is a notion of locally presentable \cv-category; see \cite{K} again. Whenever we need the projective model category structure on $[\cd,\cv]$, we have to assume either that $\cv$ satisfies the monoid axiom of \cite{SS}, or that $\cd$ is locally cofibrant, in the sense that all its hom-objects are cofibrant in \cv: see \cite[24.4]{S}. Finally, we also need to suppose that there is a cofibrant replacement functor $Q\colon\cv\to\cv$ which is enriched. But in fact this last assumption, together with the earlier assumption that the unit is cofibrant, already implies that {\em all} objects of \cv are cofibrant---see Proposition~\ref{prop:all-cofibrant}---and in this case \cd is automatically locally cofibrant, and indeed the monoid axiom follows from the assumption that \cv is a monoidal model category. 

Thus we may summarize our assumptions by saying that \cv is a combinatorial monoidal model category in which all objects are cofibrant.

Since the assumption that all objects are cofibrant is very strong, perhaps we should discuss briefly why it is needed. (This assumption was also made in \cite[Appendix~A]{L} in constructing a model structure on \vcat.) A key aspect in the theory of (enriched) locally presentable categories is that given a \cv-category \ck and a full subcategory \cg, there is an induced \cv-functor $\tilde{J}\colon\ck\to[\cg\op,\cv]$ sending an object $A\in\ck$ to the presheaf $\ck(J-,A)\colon\cg\op\to\cv$, where $J\colon\cg\to\ck$ is the inclusion. If \ck is cocomplete and \cg is closed in \ck under finite colimits, then $\tilde{J}$ will land in the locally finitely presentable category $\cm=\Lex(\cg\op,\cv)$, and one can now characterize when $\tilde{J}\colon\ck\to\cm$ is an equivalence. 

In the homotopy context, we want to replace $\tilde{J}\colon\ck\to[\cg\op,\cv]$ by a \cv-functor $\ck\to\Int[\cg\op,\cv]$ landing in the full subcategory of $[\cg\op,\cv]$ consisting of the fibrant and cofibrant objects. Since the hom-objects of \ck will be assumed to be fibrant, certainly the values of $\tilde{J}$ are fibrant, but there is no reason in general why they should be cofibrant. To rectify this, we compose $\tilde{J}$ with a cofibrant replacement functor $Q\colon[\cg\op,\cv]\fib\to\Int[\cg\op,\cv]$, but of course this $Q$ should itself be a \cv-functor. It is not hard to use an enriched form of the small object argument to construct such a \cv-functor $Q$, provided that there exists a cofibrant replacement \cv-functor $\cv\to\cv$. In an appendix to the paper, we sketch how this enriched small-object argument goes (see also \cite[24.2]{S}), as well as giving the argument, referred to above, that the existence of such a \cv-functor $Q$ along with the assumption that the unit $I$ is cofibrant implies that all objects are cofibrant. 

Although our assumptions on \cv are strong, there are nonetheless quite a few examples. Of course the classical example is \SSet. Another key example is \Cat, with the natural/categorical model structure. If $R$ is a Frobenius ring which is also a finite dimensional Hopf algebra over a field, then the category of $R$-modules with the stable model structure is an example. Another example is the category of chain complexes of comodules over a commutative Hopf algebra defined over a field, equipped with the projective model structure. All of these are described in \cite{Ho}.

An example closely related to \Cat is the cartesian closed model category \Gpd of small groupoids. The locally finitely presentable category of non-negatively graded chain complexes over a field also has a cofibrantly generated model structure in which all objects are cofibrant: a straightforward modification of the proof of \cite[Proposition~4.2.13]{Ho} shows that this is a monoidal model category under the usual tensor product.
Another source of examples is provided by {\em Cisinski model categories} \cite{C}: these are model structures on a topos, in which the cofibrations are the monomorphisms and so in particular all objects are cofibrant. Toposes are cartesian closed, and a Cisinski model category will often be a monoidal model category with respect to the cartesian monoidal structure. (The compatibility condition between monoidal structure and cofibrations always holds.)

The recent paper \cite{GuillouMay-enrichedhomotopy} also studies enrichment in homotopical settings. They characterize enriched model categories which are Quillen equivalent to enriched presheaf categories with respect to the projective model structure. Their context is more general --- in particular they do not need to require all objects of \cv to be cofibrant --- but the problem is less general, since they consider only presheaf categories rather than locally presentable ones). In particular, Theorem~\ref{thm:GuillouMay-comparison} shows that our (enriched) homotopy locally presentable categories correspond to left Bousfield localizations of presheaf categories (with the projective model structure). The paper  \cite{GuillouMay-enrichedhomotopy} also contains many general facts about enrichment in the homotopy-theoretic context, and copious references to earlier work. 

\tableofcontents

% The structure of the paper is as follows. 
% In Section~\ref{sect:equivalences} we study homotopy equivalences in \cv-categories, and in Section~\ref{sect:orthogonality} we study homotopy orthogonality. Then in Section~\ref{sect:hcolim} we study homotopy limits and colimits, 
% before turning to homotopy orthogonality classes in Section~\ref{sect:orthogonality-classes}. Finally in Section~\ref{sect:hlp} we give our main results concerning homotopy locally presentable enriched categories.

\section{Review of enriched categories}
\label{sect:review}

\subsection*{Notation}

For a morphism $f\colon X\to Y$ in a \cv-category \ck, composition with $f$ induces maps 
$$\xymatrix @R1pc {
\ck(U,X) \ar[r]^{\ck(U,f)} & \ck(U,Y) & 
\ck(Y,V) \ar[r]^{\ck(f,V)} & \ck(X,V) }$$
in \cv. In order to save space, we shall sometimes call these $f_*$ and $f^*$, respectively, if we allow ourselves to think that the context makes clear what the domains and codomains are.

For a \cv-functor $F\colon \ck\to\cl$ and objects $X,Y\in\ck$, there is an induced morphism 
$$\xymatrix{
\ck(X,Y) \ar[r]^-{F_{X,Y}} & \cl(FX,FY) }$$
which we shall sometimes simply call $F$. A \cv-functor $F\colon \ck\to\cl$ is sometimes called a {\em diagram} in \cl of shape \ck, especially if \ck is small.

\subsection*{Enriched categories and ordinary categories}

If \ck is a \cv-category, we write $\ck_0$ for the underlying ordinary category with the same objects as \ck but with morphisms from $X$ to $Y$ given by morphisms $I\to \ck(X,Y)$ in \cv. The assignment $\ck\mapsto \ck_0$ defines a 2-functor from \cv-categories to categories; this has a left adjoint sending the ordinary category \cx to the \cv-category $\overline{\cx}$ with the same objects as \cx and with hom-object $\overline{\cx}(X,Y)$ given by the copower $\cx(X,Y)\cdot I$; that is, the coproduct of $\cx(X,Y)$ copies of $I$.

By the universal property of the adjunction, if \cx is an ordinary category and \ck a \cv-category, then any ordinary functor $S\colon\cx\to\ck_0$ extends to a unique \cv-functor $\overline{S}\colon\overline{\cx}\to\ck$. In particular, 
for an ordinary category \cx, we have the ordinary functor $\cx\op\to\cv_0$ constant at the unit object $I$ of \cv, and this extends to a \cv-functor $\Delta I\colon\overline{\cx}\op\to\cv$. Limits or colimits weighted by $\Delta I$ are called {\em conical} limits or colimits. 

\subsection*{Limits and colimits}

A {\em weight} is a presheaf $\cd\op\to\cv$ on a small \cv-category. 
Given a $\cv$-category $\ck$, a diagram $S\colon\cd\to\ck$, a weight $G\colon\cd\op\to\cv$, and an object $C\in\ck$, we may form the presheaf $\ck(S,C)\colon\cd\op\to\cv$ which sends an object $D\in\cd$ to the \cv-valued hom $\ck(SD,C)$. 

The Yoneda lemma provides a bijection between morphisms $\delta\colon G\to\ck(S,C)$ in $[\cd\op,\cv]$ and morphisms $\beta\colon \ck(C,-)\to[\cd\op,\cv](G,\ck(S,-))$ in $[\ck,\cv]$. We say that $\delta$ exhibits $C$ as the (weighted) colimit $G*S$ if the corresponding $\beta$ is invertible. 

A special case is where $\cd$ is the unit \cv-category \ci consisting of a single object $0$ with $\ci(0,0)=I$; then a diagram $S\colon\ci\to\ck$ just consists of an object $S\in\ck$, while a weight consists of an object $X\in\cv$. The corresponding weighted colimit, usually written $X\cdot S$, and called a {\em copower}, is defined by a natural isomorphism 
$$\ck(X\cdot S,-) \cong \cv(X,\ck(S,-)).$$

A limit in \ck is the same as a colimit in $\ck\op$, but typically one writes in terms of a diagram $S\colon\cd\to\ck$ and weight $\cd\to\cv$ (a presheaf on $\cd\op$). 

\subsection*{Local presentability}

A symmetric monoidal closed category \cv is said to be {\em locally $\lambda$-presentable as a closed category} \cite{K} if the underlying ordinary category is locally $\lambda$-presentable in the usual sense, and the full subcategory of $\lambda$-presentable objects is closed under tensoring and contains the unit. 

In this case there is a good theory of \cv-enriched locally $\lambda$-presentable categories \cite{K}; a \cv-category \ck is locally $\lambda$-presentable if and only if it is a full reflective \cv-category of a presheaf category $[\cd\op,\cv]$, closed under $\lambda$-filtered colimits. 

Any symmetric monoidal closed category \cv which is locally presentable as a category is locally $\lambda$-presentable as a closed category for some $\lambda$, and then also for all larger values of $\lambda$: see \cite{vcat}.

\subsection*{Model \cv-categories}

We suppose that \cv is a monoidal model category in the sense used in \cite{L}, which includes the assumption that the unit $I$ is cofibrant, and that \cv satisfies the monoid axiom \cite{SS}.

A model \cv-category \cite{Ho} is a complete and cocomplete \cv-category \cm, with a model structure on the underlying ordinary category $\cm_0$, subject to a compatibility condition asserting that if $j\colon X\to Y$ is a cofibration in \cv and $f\colon A\to B$ a cofibration in \cm then the induced map $j\square f$ out of the pushout in 
$$\xymatrix{
X\cdot A \ar[r]^{j\cdot A} \ar[d]_{X\cdot f} & Y\cdot A \ar[d]  \ar@/^1pc/[ddr]^{Y\cdot f} \\
X\cdot B \ar[r] \ar@/_1pc/[drr]_{j\cdot B} & P \ar@{.>}[dr]^{j\square f} \\
&& Y\cdot B }$$
is a cofibration; and a trivial cofibration if either $j$ or $f$ is one. 

For a small \cv-category \cd, we regard the presheaf category $[\cd\op,\cv]$ as a model \cv-category under the projective model structure. 

Suppose that \cm is a model \cv-category. A morphism $f\colon A\to B$ in \cm may be identified with a morphism $I\to \cm(A,B)$. Since \cm and \cv each have model structures, we could consider either the homotopy relation on morphisms $A\to B$ defined using the model structure on \cm, or the homotopy relation on morphisms $I\to\cm(A,B)$ defined using the model structure of \cv. The following easy result helps to resolve this potential ambiguity:

\begin{propo}\label{prop:two-notions}
  If $A$ is cofibrant and $B$ is fibrant in \cm, then the two notions of homotopy agree. 
\end{propo}

\proof
Let $f,g\colon A\to B$ be morphisms, and write $f',g'$ for the corresponding $I\to\cm(A,B)$.
Factorize the codiagonal $I+I\to I$ as a cofibration $(i~j)\colon I+I\to J$ followed by a weak equivalence $w\colon J\to I$; it follows that $i$ and $j$ are trivial cofibrations. Since $I$ and $A$ are cofibrant, $A+A$ is cofibrant; also $(i~j)\cdot A\colon A+A\to J\cdot A$ is a cofibration with $i\cdot A$ and $j\cdot A$ trivial cofibrations, and so $w\cdot A\colon J\cdot A\to A$ is a weak equivalence. 

Then $f'$ is homotopic to $g'$ if and only if the induced $(f'~g')\colon I+I\to \cm(A,B)$ factorizes through $(i~j)$, say as $h'\colon J\to \cm(A,B)$. And $f$ is homotopic to $g$ if and only if the induced $(f~g)\colon A+A\to B$ factorizes through $(i~j)\cdot A$, say as $h\colon J\cdot A\to B$. But to give such an $h$ is equivalently to give an $h'$, by the universal property of the copower $J\cdot A$. Thus $f$ is homotopic to $g$ if and only if $f'$ is homotopic to $g'$. 
\endproof

\section{Homotopy equivalences}\label{sect:equivalences}
 
Given a monoidal model category $\cv$, we have a monoidal structure on $\Ho\cv$, for which the canonical functor $P\colon\cv\to\Ho\cv$ is strong 
monoidal. Since the hom-functor $\Ho\cv(I,-)\colon\Ho\cv\to\Set$ is also monoidal, so is the composite $U:=\Ho\cv(I,P-)\colon\cv\to\Set$. On the other hand, there is also the monoidal functor $\cv(I,-)\colon\cv\to\Set$, and $P$ induces a monoidal natural transformation $p\colon\cv(I,-)\to U$ whose component at $A\in\cv$ is the function $p\colon\cv(I,X)\to\Ho\cv(I,X)$ given by applying $P$.

The following definition also appears in \cite[A.3.2.9]{L}, although we have been more explicit about the role of $p\colon\cv(I,-)\to U$.

\begin{defi}\label{defn:hoK}
Let $\cv$ be a monoidal model category and $\ck$ a $\cv$-category. The {\em homotopy category} $\ho\ck$ of $\ck$ has the same objects as $\ck$, 
and $\ho\ck(A,B)=U(\ck(A,B))$. There is an induced functor $p_*$ from the underlying ordinary category $\ck_0$ of \ck to $\ho\ck$, sending a morphism $f\colon I\to\ck(A,B)$ to $p(f)$.
 
A morphism $f\colon A\to B$ in a $\cv$-category $\ck$ is called a {\em homotopy equivalence} if its image in $\ho\ck$ is invertible.
\end{defi}

\begin{rem}\label{rmk:ho-Ho}
For a model $\cv$-category $\cm$, we now have the standard homotopy category $\Ho\cm$ of the underlying ordinary category $\cm_0$ of $\cm$, and the homotopy category $\ho\cm$ defined 
using the enrichment, and these need not agree. But if $\Int\cm$ is the full subcategory of $\cm$ consisting of the fibrant and cofibrant objects,
then $\ho(\Int\cm)$ is equivalent to $\Ho(\cm)$, thanks to Proposition~\ref{prop:two-notions}. 

Since the passage from $\ck$ to $\ho\ck$ is functorial, a $\cv$-functor $\ck\to\cl$ sends homotopy equivalences to homotopy equivalences. Furthermore, if a \cv-functor $\ck\to\cl$ is fully faithful then so is the induced functor $\ho\ck\to\ho\cl$; thus fully faithful \cv-functors reflect homotopy equivalence. 
\end{rem}
 
\begin{defi}\label{defn:fibrant}
A $\cv$-category $\ck$ is said to be {\em fibrant} if each hom-object $\ck(A,B)$ is fibrant in $\cv$. 
\end{defi}

\begin{rem}\label{rmk:loc-fib}
These are also called {\em locally fibrant} \cite[A.3.2.9]{L}, following the usage that an enriched category is ``locally $P$'' if its hom-objects are $P$. The name fibrant was also used in \cite{R1} in the case $\cv=\SSet$, and is justified by the fact that, in those cases where a model structure on \vcat has been defined, the fibrant objects are precisely the fibrant \cv-categories in our sense.
\end{rem}

\begin{exam}
  If \cm is a locally presentable \cv-category with a \cv-enriched model structure, then $\cm(A,B)$ is fibrant in \cv whenever $A$ is cofibrant in \cm and $B$ is fibrant in \cm. Thus $\Int\cm$ is a fibrant \cv-category.
\end{exam}

\begin{rem}\label{rmk:hty-equiv-symmetry}
Let \cm be an ordinary model category, with homotopy category $P\colon\cm\to\Ho\cm$. If $X$ is cofibrant and $Y$ is fibrant, then $\Ho\cm(PX,PY)$ may be identified with the quotient of $\cm(X,Y)$ by the homotopy relation. In particular, $P_{X,Y}\colon\cm(X,Y)\to\Ho\cm(PX,PY)$ is surjective. If now \ck is a fibrant \cv-category, then for any two objects $A,B\in\ck$ we have $I$ cofibrant and $\ck(A,B)$ fibrant in \cv; thus 
$$\xymatrix{
\ck_0(A,B) = \cv(I,\ck(A,B)) \ar[r]^-{P} & \Ho\cv(I,\ck(A,B)) = \ho\ck(A,B) }$$
is surjective, and any morphism in $\ho\ck$ is induced by one in $\ck_0(A,B)$.
This means, in particular, that if \ck is a fibrant \cv-category, then there exists a homotopy equivalence $A\to B$ if and only if there exists a homotopy equivalence $B\to A$. Thus ``homotopy equivalence'' defines an equivalence relation on the objects of a fibrant \cv-category. In this case we shall sometimes write $A\simeq B$. 
\end{rem}

\begin{defi}[J. H. Smith] \label{def:lambda-combinatorial}
  A model category is {\em $\lambda$-combinatorial}, for a regular cardinal $\lambda$, if it is locally $\lambda$-presentable as a category, cofibrantly generated as a model category, and the generating cofibrations and trivial cofibrations may be chosen to have $\lambda$-presentable domains and codomains. It is combinatorial if it is $\lambda$-combinatorial for some $\lambda$ This will be the case for some $\lambda$ if and only if it is cofibrantly generated and locally presentable; it will then also be $\mu$-combinatorial whenever $\mu>\lambda$.
\end{defi}

\begin{rem}\label{rmk:combinatorial}
We shall use the following facts about $\lambda$-combinatorial model categories:
\begin{itemize}
\item the cofibrant and fibrant replacement functors preserve $\lambda$-filtered colimits;
\item the weak equivalences are closed under $\lambda$-filtered colimits;
\item the fibrant objects are closed under $\lambda$-filtered colimits.
\end{itemize}
Proofs of the first two can be found in \cite[2.3]{D} or \cite[3.1]{R}, while the third is really a general fact about injectivity classes, and was proved in that context in \cite[4.7]{AR}.
\end{rem}

\begin{propo}\label{prop:hty-equiv}
Let $\cv$ be a combinatorial monoidal model category satisfying the monoid axiom, $\ck$ a fibrant $\cv$-category and $f\colon A\to B$ a morphism 
in $\ck$. Then the following conditions are equivalent:
\begin{enumerate}
\item[(i)] $f$ is a homotopy equivalence.
\item[(ii)] $\ck(C,f)$ is a weak equivalence in $\cv$, for all $C\in\ck$.
\item[(iii)] $\ck(f,C)$ is a weak equivalence in $\cv$, for all $C\in\ck$.
\item[(iv)] $\ck(C,f)$ is a weak equivalence in $\cv$ for $C$ equal to $A$ or $B$;
\item[(v)] $\ck(f,C)$ is a weak equivalence in $\cv$ for $C$ equal to $A$ or $B$.
\end{enumerate}
\end{propo}
\begin{proof}
Clearly (i) implies all the other conditions, since representable functors send homotopy equivalences in $\ck$ or $\ck\op$ to homotopy equivalences 
in $\cv$, and homotopy equivalences in $\cv$ are weak equivalences.  Even more clearly (ii) implies (iv) and (iii) implies (v). If we can prove 
that (iv) implies (i), then dually (v) will imply (i), and so all conditions will be equivalent.

Suppose then that $\ck(C,f)$ is a weak equivalence for $C$ equal to $A$ or $B$. Let $\cc$ be the full subcategory of $\ck$ with objects $A$ and $B$, 
and consider the projective model structure on $[\cc\op,\cv]$. The inclusion  $J\colon\cc\to\ck$ induces a $\cv$-functor 
$\tilde{J}\colon\ck\to[\cc\op,\cv]$ with $\tilde{J}J$ equal to the Yoneda functor $Y$. Now $\tilde{J}f\colon\tilde{J}A\to\tilde{J}B$ is a weak equivalence 
in $[\cc\op,\cv]$ by assumption, but its domain and codomain are the representables $\cc(-,A)$ and $\cc(-,B)$ which are fibrant and cofibrant,
thus $\tilde{J}f$ is in fact a homotopy equivalence. But that means $Yf$ is a homotopy equivalence, whence $f$ is a homotopy equivalence in $\cc$, and 
so also in $\ck$. 
\end{proof}

\section{Homotopy orthogonality}\label{sect:orthogonality}

\begin{defi}\label{defn:hty-orthog}
Let $\cv$ be a monoidal model category, $\ck$ a $\cv$-category, and $f\colon A\to B$ a morphism in $\ck$. Then an object $K$ in $\ck$  is called
{\em homotopy orthogonal} to $f$ if $\ck(f,K)$ is a weak equivalence. 
\end{defi}

While homotopy orthogonality is clearly some sort of homotopy version of orthogonality, it also resembles injectivity in some ways: see Proposition~\ref{prop:Fhorn} for example.
In the terminology of \cite{LR1}, such a homotopy orthogonal object would be called $f$-injective over the weak equivalences. An object $K$ is homotopy orthogonal to a class $\cf$ of morphisms 
if it is homotopy orthogonal to each $f\in\cf$. The class of all objects homotopy orthogonal to $\cf$ is denoted by $\cf$-$\Inj$. 
{\em Small homotopy orthogonality classes} are defined as classes $\cf$-$\Inj$ where $\cf$ is a set. Without this limitation, we speak about 
{\em homotopy orthogonality classes}.

By Proposition~\ref{prop:hty-equiv}, a morphism $f$ is a homotopy equivalence if and only if {\em every} object $K$ is homotopy orthogonal to $f$. 

\begin{lemma}\label{lemma:hty-orthog-invariance}
Let $\ck$ be a fibrant $\cv$-category, $f\colon A\to B$ a morphism in $\ck$, and $K$ and $L$ homotopy equivalent objects of \ck. Then $K$ is homotopy orthogonal to $f$ if and only if $L$ is so. 
\end{lemma}
\begin{proof}
Suppose that $K$ is homotopy orthogonal to $f$, and let $h\colon K\to L$ be a homotopy equivalence. In the commutative square
$$
\xymatrix@C=4pc@R=3pc{
\ck(B,K) \ar[r]^{\ck(f,K)} \ar[d]_{\ck(B,h)} &
\ck(A,K)\ar [d]^{\ck(A,h)}\\
\ck(B,L) \ar[r]_{\ck(f,L)}& \ck(A,D)
}
$$
the vertical morphisms are weak equivalences by Proposition~\ref{prop:hty-equiv}, thus $\ck(f,K)$ is a weak equivalence if and only if $\ck(f,L)$ is one. 
%by assumption, whence $\ck(f,L)$ too is a weak equivalence.
\end{proof}

\begin{lemma}\label{lemma:hty-orthog-comp}
Let $\ck$ be a fibrant $\cv$-category, $f\colon A\to B$ a morphism in $\ck$, and $f= gh$ where $h$ is a homotopy equivalence. Then an object $K$ is homotopy orthogonal to $f$ if and only if it homotopy orthogonal to $g$.
\end{lemma}
\begin{proof}
Since 
$$
\ck(f,K)=\ck(h,K)\ck(g,K)
$$
and $\ck(h,K)$ is a weak equivalence by Proposition~\ref{prop:hty-equiv}, $\ck(f,K)$ is a weak equivalence if and only if $\ck(g,K)$
is a weak equivalence.
\end{proof}

\begin{defi}\label{defn:hty-refl}
Let $\cl$ be a full sub-$\cv$-category of a fibrant $\cv$-category $\ck$. We say that $\cl$ is {\em homotopy reflective} in $\ck$ if,
for each $K$ in $\ck$, there is a morphism $\eta_K\colon K\to K^\ast$ with $K^\ast$ in $\cl$ such that each $L$ in $\cl$ is homotopy orthogonal to $\eta_K$.
\end{defi}

Homotopy reflective full subcategories coincide with subcategories which are, in the sense of \cite{LR1}, weakly reflective with respect to the weak equivalences. 

A locally presentable model category $\cm$ is called {\em tractable} \cite{Ba} if both cofibrations and trivial cofibrations are cofibrantly generated 
by a set of morphisms between cofibrant objects. Of course, every tractable model category is combinatorial.  

\begin{theo}\label{thm:orthog-refl}
Let $\cv$ be a tractable monoidal model category and $\cm$ a tractable left proper model $\cv$-category. Then each small homotopy orthogonality class in $\Int\cm$ is homotopy reflective.
\end{theo}
\begin{proof}
Let $\cf$ be a set of morphisms in $\Int\cm$. Since weak equivalences in $\Int\cm$ are homotopy equivalences in the sense of Definition~\ref{defn:hoK}, 
we can use Lemma~\ref{lemma:hty-orthog-comp} and assume that $\cf$ consists of cofibrations in $\Int\cm$. An object $K$ in $\Int\cm$ is $\cf$-injective if and only if 
it is $\cf$-local; that is, if and only if it is fibrant in the $\cf$-localized $\cv$-model category structure on $\cm$ (see \cite{Ba}). Thus homotopy reflections $\eta_K\colon K\to K^\ast$ are given by fibrant replacements in this model category.
\end{proof}

\begin{rem}\label{rmk:orthog-and-localization}
Let $\cv$ be a tractable monoidal model category and $\cm$ a left proper tractable model $\cv$-category. A consequence of the relationship between 
small homotopy orthogonality classes in $\Int\cm$ and enriched left Bousfield localizations in $\cm$ used in the proof above is that each small 
homotopy orthogonality class $\cf$-$\Inj$ in $\Int\cm$ is $\Int\cn$ for some combinatorial model $\cv$-category $\cn$; in particular, we could take $\cn$ to be the $\cf$-localized model $\cv$-category.
\end{rem}

Let \cm be a cofibrantly generated model \cv-category. 
If $X\in\cv$ and $A\in\cm$, recall from Section~\ref{sect:review} that the {\em copower} $X\cdot A\in\cm$,  is defined by the universal property $\cm(X\cdot A,B)\cong \cv(X,\cm(A,B))$. If $i\colon X\to Y$ is a generating cofibration in \cm, and $f\colon A\to B$ is a morphism in \cm, we can form the pushout $P_{i,f}$ as in the diagram below, and the induced map $i\square f\colon P_{i,f}\to Y\cdot B$, called the {\em pushout-product} of $i$ and $f$.
$$
\xymatrix { %@C=4pc@R=3pc{
X\cdot A \ar[r]^{i\cdot A} \ar[d]_{X\cdot f} &
Y\cdot A  \ar[d] \ar@/^1pc/[ddr]^{Y\cdot f}\\
X\cdot B \ar@/_1pc/[drr]_{i\otimes B} \ar[r] & P_{i,f} \ar@{.>}[dr]^{i\square f} \\
&& Y\cdot B
}
$$
Such a map $i\square f$ is called an {\em $f$-horn}, and if $\cf$ is a class of morphisms then we denote by $\Hor(\cf)$ the class of $f$-horns, for all $f\in \cf$. Part of the definition of model \cv-category is that if $f$ is a cofibration or trivial cofibration, then so is each $f$-horn.

Recall that $\Int\cm$ denotes the full subcategory of \cm consisting of those objects which are both fibrant and cofibrant. 

\begin{propo}\label{prop:Fhorn}
Let \cf be a set of cofibrations in $\Int\cm$. An object $K\in\Int\cm$ is homotopy orthogonal to \cf if and only if it is injective in $\cm_0$ to all \cf-horns.
\end{propo}
\proof
Since each $f$ is a cofibration, each $\Int\cm(f,K)$ is a fibration. Thus $K$ will be homotopy orthogonal to the $f$ if and only if each $\Int\cm(f,K)$ is a trivial fibration; in other words, each $\Int\cm(f,K)$ has the right lifting property with respect to each generating cofibration $i\colon X\to Y$. But this is equivalent to $K$ being injective in $\cm_0$ with respect to the \cf-horns.
\endproof

\begin{theo}\label{vopenka}
Let $\cv$ be a tractable monoidal model category and $\cm$ a tractable left proper model $\cv$-category. Assuming Vop\v enka's principle, each
homotopy orthogonality class in $\Int\cm$ is a small homotopy orthogonality class in $\Int\cm$.
\end{theo}
\begin{proof}
Let $\cf$ be a class of morphisms in $\Int\cm$. Then $\cf$ is the union of an increasing chain of subsets $\cf_i$ indexed by ordinals. Let $\cw_i$
denote the class of weak equivalences in the $\cf_i$-localized model structure on $\cm_0$. Following \cite[2.3]{RT}, there exists the $\cf$-localized model
structure whose weak equivalences are $\cw=\cup\cw_i$. Let $\ci$ be a set of generating cofibrations in $\cm$. Following \cite[2.2]{RT} and \cite[1.7]{Bk}, trivial cofibrations in the $\cf$-localized model structure are generated by a set $\cj$ dense between $\ci$ and $\cw$. Since $\cj$ is dense between $\ci$
and some $\cw_i$, trivial cofibrations in the $\cf_i$-localized model structure coincide with those in the $\cf$-localized one (see \cite[1.7]{Bk} again).
Hence $\cf$-local objects coincide with $\cf_i$-local objects and thus $\cf$-$\Inj=\cf_i$-$\Inj$.
\end{proof}

\begin{rem}\label{compare}
This generalizes the fact that, under Vop\v enka's principle, any orthogonality class in a locally presentable category is a small orthogonality class \cite[6.24]{AR}. On the other hand, the corresponding statement for injectivity classes is false.  For example, complete lattices form an injectivity class in posets which is not 
a small-injectivity class (see \cite[Example 4.7]{AR}). 
\end{rem}

\section{Homotopy weighted colimits}\label{sect:hcolim}

Recall that, given a $\cv$-category $\ck$, a colimit $G\ast S$ of a diagram $S\colon\cd\to\ck$ weighted by $G\colon\cd\op\to\cv$ is defined by a natural isomorphism
$$
\ck(G\ast S,-)\cong[\cd\op,\cv](G,\ck(S,-)).
$$
If $\cv$ is a monoidal model category satisfying the monoid axiom, the $\cv$-category $[\cd\op,\cv]$ may be equipped with the projective
model $\cv$-category structure. A cofibrant object in $[\cd\op,\cv]$ will be called a {\em cofibrant weight}, and we shall use the term ``cofibrant colimit'' to mean a weighted colimit for which the weight is cofibrant. We write $\Phi(\cd)$ for the full sub-$\cv$-category of $[\cd\op,\cv]$ consisting of cofibrant weights. We are going to show that $\Phi(\cd)$ is closed under {\em cofibrant colimits}, so that the class of cofibrant weights is {\em saturated} (or {\em closed} in the original terminology of \cite{Albert-Kelly}). It then follows that $\Phi(\cd)$ is the free cocompletion of $\cd$ under cofibrant colimits.

\begin{propo}\label{prop:saturated}
The class $\Phi$ of cofibrant weights is saturated.
\end{propo}
\begin{proof}
Let $\cc$ be a small \cv-category. Suppose that $S\colon\cd\to[\cc\op,\cv]$ takes values in $\Phi(\cc)$, and that $G\colon\cd\op\to\cv$ is cofibrant. We must show that 
$G\ast S$ is in $\Phi(\cc)$. If $p\colon A\to B$ is a trivial fibration in $[\cc\op,\cv]$, and $u\colon G\ast S\to B$, we must find a lifting 
of $u$ through $p$. To give $u$ is equally to give $u'\colon G\to [\cc\op,\cv](S,B)$ in $[\cd\op,\cv]$, and to find a lifting of $u$ 
is equivalent to finding a lifting of $u'$ through $[\cc\op,\cv](S,p)\colon[\cc\op,\cv](S,A)\to[\cc\op,\cv](S,B)$. Since $G$ is cofibrant, 
this will be possible provided that $[\cc\op,\cv](S,p)$ is a trivial fibration in $[\cd\op,\cv]$; that is, provided that 
$[\cc\op,\cv](SD,p)\colon[\cc\op,\cv](SD,A)\to[\cc\op,\cv](SD,B)$ is a trivial fibration in $\cv$ for each $D\in\cd$. But $p\colon A\to B$ is a trivial fibration, and each 
$SD$ is cofibrant, so this holds because $[\cc\op,\cv]$ is a model $\cv$-category.
\end{proof}

\begin{defi}\label{defn:hty-colimit}
Let $\cv$ be a monoidal model category satisfying the monoid axiom, $\ck$ a fibrant $\cv$-category, $S\colon\cd\to\ck$ a diagram, 
and $G\colon\cd\op\to\cv$ a cofibrant weight.
% Let $G_c\colon\cd\op\to\cv$ be a cofibrant replacement of $G$ in the projective model category structure. 
Then a {\em homotopy colimit} 
of $S$ weighted by $G$ is an object $G\ast_h S$ equipped with a natural transformation
$$
\beta\colon\ck(G\ast_h S,-)\to[\cd\op,\cv](G,\ck(S,-))
$$
whose components are weak equivalences.
\end{defi} 

By the (enriched) Yoneda lemma, the natural transformation $\beta$ in the definition of homotopy colimit corresponds to a cocone $\delta\colon G\to \ck(S,G*_h S)$.
We now group together a list of facts about the existence and uniqueness of homotopy colimits.

\begin{propo}\label{prop:hcolimit-uniqueness}
  Let \cv be a monoidal model category satisfying the monoid axiom, let \ck be a fibrant \cv-category, let $G,H\colon\cd\op\to\cv$ be cofibrant weights, and let $S,T\colon \cd\to\ck$ be \cv-functors.
  \begin{enumerate}
\item If  the weighted colimit $G*S$ exists, then it is a homotopy colimit $G*_h S$.
\item Weighted homotopy colimits are determined up to homotopy equivalence.
\item If $\phi\colon G\to H$ is a weak equivalence, then a homotopy colimit $H*_h S$ exists if and only if  $G*_h S$ does so, and they then agree.
\item If $\delta\colon G\to\ck(G*_h S,S)$ exhibits $G*_h S$ as the homotopy colimit, then so does any morphism in $[\cd\op,\cv]$ which is homotopic to $\delta$.
\item If $\psi\colon S\to T$ is a pointwise homotopy equivalence, then a homotopy colimit $G*_h S$ exists if and only if $G*_h T$ does so, and they then agree.
  \end{enumerate}
\end{propo}

\proof 
\noindent (1)
 If the actual colimit $G*S$ exists then we have a natural isomorphism $\beta$, not just a natural weak equivalence.

\noindent (2)
  Let $K_1$ and $K_2$ be homotopy colimits of $S$ weighted by $G$. 
Let $J\colon\cx\to\ck$ be a small full subcategory of $\ck$ containing $K_1$, $K_2$, and the image of $S$. The induced morphisms 
$$\xymatrix @R0pc {
\cx(K_1,-) \ar@{=}[r] &  \ck(K_1,J) \ar[r]^-{\beta_1} &  [\cd\op,\cv](G,\ck(S,J)) \\
\cx(K_2,-) \ar@{=}[r] & \ck(K_2,J) \ar[r]^-{\beta_2} &  [\cd\op,\cv](G,\ck(S,J)) }
$$
are pointwise weak equivalences in $[\cx,\cv]$, so $\cx(K_1,-)$ and $\cx(K_2,-)$ are weakly equivalent in $[\cx,\cv]$; but they are also cofibrant and fibrant objects, and so they are homotopy equivalent. Thus $K_1$ and $K_2$ are homotopy equivalent in $\cx$, and so also in $\ck$. This proves that any two choices of homotopy colimit are homotopy equivalent. Similarly, any object homotopy equivalent to a homotopy colimit can itself be used as a homotopy colimit.

\noindent (3)
Since $\phi\colon G\to H$ is a weak equivalence between cofibrant objects and $\ck(S,X)$ is fibrant for all $X\in\ck$, also  $[\cd\op,\cv](\phi,\ck(S,X))$ is a weak equivalence for all $X\in\ck$. Thus if $H*_h S$ exists then we have weak equivalences 
$$\xymatrix @C1pc {
\ck(H*_h S,X) \ar[r] & [\cd\op,\cv](H,\ck(S,X)) % \ar[rrrr]^{[\cd\op,\cv](\phi_c,\ck(D,-))} &&&
\ar[r] 
& 
[\cd\op,\cv](G,\ck(S,X)) }$$
natural in $X$, 
and so $H*_h D$ also serves as a homotopy colimit $G*_h D$. 

The converse is more delicate. We need to show that $G*_h S$ will serve as $H*_h S$.
Form the coproduct $G+H$ and the map $G+H\to H$ induced by $\phi$ and $1_H$, and factorize it as a cofibration $\alpha\colon G+H\to K$ followed by a trivial fibration $\tau\colon K\to H$. Restricting $\alpha$ to $G$ and $H$, we obtain a factorization $\phi=\tau\rho$ of $\phi$ as a trivial cofibration $\rho\colon G\to K$ followed by a trivial fibration $\tau\colon K\to H$, as well as a trivial cofibration $\sigma\colon H\to K$ which is also a section of $\tau$. By the first part of the proof and the existence of $\sigma$, we know that if $K*_h S$ exists it will also serve as $H*_h S$. Thus it will suffice to show that if $G*_h S$ exists, then it will serve as $K*_h S$. Suppose then that 
$$\xymatrix{
\ck(G*_h S,-) \ar[r]^-{\beta} & [\cd\op,\cv](G,\ck(S,-)) }$$
exhibits $G*_h S$ as a homotopy colimit. For each $X\in\ck$ we have the presheaf $\ck(S,X)\colon\cd\op\to\cv$, and this is fibrant in $[\cd\op,\cv]$. On the other hand, $\rho\colon G\to K$ is a trivial cofibration. It follows that 
$$\xymatrix @C6pc {
[\cd\op,\cv](K,\ck(S,X)) \ar[r]^{[\cd\op,\cv](\rho,\ck(S,X))} & [\cd\op,\cv](G,\ck(S,X)) }$$
is a trivial fibration. But this was true for all $X\in\ck$, and so now in the diagram
$$\xymatrix @C3pc {
\ck(G*_h S,-) \ar@{.>}[r] \ar[dr]_{\beta} & [\cd\op,\cv](K,\ck(S,-)) \ar[d]^{[\cd\op,\cv](\rho,\ck(S,-))} \\
& [\cd\op,\cv](\ck(G,\ck(S,-))) }$$
the vertical arrow is a trivial fibration and $\ck(G*_h S,-)$ is cofibrant, and so there exists a lift, displayed in the diagram as a dotted arrow, and this will be a pointwise weak equivalence since the other two morphism are so.

\noindent (4)
Suppose that $\delta\colon G\to \ck(C,S)$ exhibits $C$ as the homotopy colimit $G*_h S$, and that $\delta'\colon G\to\ck(C,S)$ is homotopic to $\delta$. Since $G$ is cofibrant and $\ck(C,S)$ is fibrant, there exist trivial cofibrations $i,j: G\to G'$ and a morphism $\gamma$ making the diagram 
$$\xymatrix{
G \ar[r]^i \ar[dr]_{\delta} & G' \ar[d]^{\gamma} & G \ar[l]_{j} \ar[dl]^{\delta'} \\
& \ck(C,S) }$$
commute. In the diagram 
$$\xymatrix{
&& [\cd\op,\cv](G,\ck(S,A)) \\
\ck(C,A) \ar[r]^-{\ck(S,-)} & [\cd\op,\cv](\ck(S,C),\ck(S,A)) \ar[r]^{\gamma^*} \ar[ur]^{\delta^*} \ar[dr]_{(\delta')^*} & [\cd\op,\cv](G',\ck(S,A)) \ar[d]^{j^*} \ar[u]_{i^*} \\
&& [\cd\op,\cv](G,\ck(S,A)) }$$
the vertical map $i^*$ is a weak equivalence because $\ck(S,A)$ is fibrant and $i\colon G\to G'$ is a trivial cofibration; similarly $j^*$ is a weak equivalence. Thus the composite $\delta^*\ck(S,-)$ is a weak equivalence if and only if the composite $(\delta')^*\ck(S,-)$ is one. 

\noindent (5)
 If $\psi\colon S\to T$ is a pointwise homotopy equivalence, then 
$$\xymatrix{ \ck(T,X)\ar[r]^{\ck(\psi,X)} & \ck(S,X) }$$ 
is a weak equivalence in $[\cd\op,\cv]$ between fibrant objects, and so 
$$\xymatrix @C6pc { [\cd\op,\cv](G,\ck(T,X)) \ar[r]^{ [\cd\op,\cv](G,\ck(\psi,X))} & [\cd\op,\cv](G,\ck(S,X)) }$$ 
is a weak equivalence in \cv between fibrant objects, for any $X\in\ck$. Thus if the homotopy colimit $G*_h T$ exists, then we have a composite weak equivalence 
$$\xymatrix{
\ck(G*_h T,X) \ar[r] & [\cd\op,\cv](G,\ck(T,X)) \ar[r] & [\cd\op,\cv](G,\ck(S,X)) }$$ 
natural in $X$, and so $G*_h T$ serves as $G*_h S$.

Once again, the converse is more delicate. Suppose that $\delta\colon G\to \ck(S,C)$ exhibits $C$ as the homotopy colimit $G*_h S$. Since $\psi\colon S\to T$ is a pointwise homotopy equivalence, the morphism $\psi^*\colon \ck(T,C)\to\ck(S,C)$ in $[\cd\op,\cv]$ induced by composition with $\psi$ is a pointwise weak equivalence between fibrant objects. Since $G$ is cofibrant, $\delta\colon G\to \ck(S,C)$ can be lifted, up to homotopy, through $\psi^*$, say by $\gamma\colon G\to \ck(T,C)$. But $\psi^*\gamma$ is homotopic to $\delta$, and so by the previous part it also exhibits $C$ as the homotopy colimit $G*_h S$. Thus we may as well replace $\delta$ by $\psi^*\gamma$, and then regard $\gamma$ as a genuine lifting of $\delta$ through $\psi^*$. 

Now consider the diagram
$$\xymatrix{
\ck(C,A) \ar[r]^-{\ck(S,-)} \ar[d]_{\ck(T,-)} & [\cd\op,\cv](\ck(S,C),\ck(S,-)) \ar[d]^{(\psi^*)^*} \\
[\cd\op,\cv](\ck(T,C),\ck(T,-)) \ar[r]^-{(\psi^*)_*} \ar[d]_{\gamma^*} & [\cd\op,\cv](\ck(T,C),\ck(S,-)) \ar[d]^{\gamma^*} \\
[\cd\op,\cv](G,\ck(T,A)) \ar[r]^{(\psi^*)_*} & [\cd\op,\cv](G,\ck(S,A)). }$$
The upper composite is the weak equivalence $\delta^*\ck(S,-)$, thus the lower composite is also a weak equivalence. But the bottom horizontal arrow $(\psi^*)_*$ is a weak equivalence because $G$ is cofibrant and $\psi^*\colon\ck(T,A)\to\ck(S,A)$ is a weak equivalence between fibrant objects. Thus the composite $\gamma^*\ck(T,-)$ is a weak equivalence, and so $\gamma$ exhibits $C$ as the homotopy colimit $G*_h T$.
\endproof

\begin{rem}\label{rmk:non-cofibrant-weights}
  We have defined homotopy weighted colimits only for cofibrant weights, but since the homotopy weighted colimit depends on the weight only up to weak equivalence, there is no danger in defining the homotopy colimit of $S$ weighted by an arbitrary $G$ to be the homotopy colimit weighted by the cofibrant replacement $QG$ of $G$. We shall do this when convenient. 
\end{rem}

Dually, given a fibrant $\cv$-category $\ck$, the homotopy limit $\{G,S\}_h$ of a diagram $S\colon\cd\to\ck$ weighted by $G\colon\cd\to\cv$ is defined by a natural transformation
$$
\beta\colon\ck(-,\{G,S\}_h)\to[\cd,\cv](G,\ck(-,S))
$$
whose components are weak equivalences. 

All that was said about homotopy colimits applies to homotopy limits. In particular, the natural transformation $\beta$ corresponds to a cone
$$
\delta\colon G\to\ck(\{G,S\}_h,S).
$$

\begin{theo}\label{thm:colimit-existence}
Let $\cv$ be a monoidal model category satisfying the monoid axiom and let $\cm$ be a model $\cv$-category. Then $\Int\cm$ has weighted homotopy colimits
and weighted homotopy limits.
\end{theo}
\begin{proof}
I. First we show that $\Int\cm$ has weighted homotopy colimits. 
Let $G\colon\cd^{op}\to\cv$ be a cofibrant weight and $S\colon\cd\to\Int\cm$ a diagram. Form the colimit $G\ast S$ in $\cm$. First we prove that $G\ast S$ is cofibrant; this follows closely the proof of Proposition~\ref{prop:saturated}, which is essentially a special case. Let $p\colon A\to B$ be a trivial fibration in $\cm$. We must show that every morphism $f\colon G\ast S\to B$ lifts through $p$. But to give such an $f$ is equivalently to give $f'\colon G\to\cm(S,B)$ in $[\cd\op,\cv]$, and to give a lifting 
of $f$ through $p$ is to give a lifting of $f'$ through $\cm(S,p)\colon\cm(S,A)\to\cm(S,B)$. Since $G$ is cofibrant, this will exist provided 
that $\cm(S,p)$ is a trivial fibration in $[\cd\op,\cv]$; that is, provided that $\cm(SD,p)\colon\cm(SD,A)\to\cm(SD,B)$ is a trivial fibration for every $D\in\cd$. 
But $SD$ is cofibrant by assumption, and $p\colon E\to B$ is a trivial fibration, so $\cm(SD,p)$ is a trivial fibration since $\cm$ is a model $\cv$-category. 

Now, take a fibrant replacement $r\colon G\ast S\to R(G\ast S)$ via a trivial cofibration $r$. We have natural transformations
$$
\xymatrix{
\Int\cm(R(G\ast S),A) \ar@{=}[d] & [\cd\op,\cv](G,\Int\cm(S,A)) \\
\cm(R(G\ast S),A) \ar[r]^-{\cm(r,A)} & \cm(G\ast S,A) \cong [\cd\op,\cv](G,\cm(S,A)) 
\ar@{=}[u]    }
$$
and so if $\cm(r,A)$ is a weak equivalence for all $A\in\Int\cm$, then the replacement $R(G\ast S)$ will be the desired homotopy colimit $G\ast_h S$ 
in $\Int\cm$. But $r$ is a trivial cofibration and $A$ is fibrant, so $\cm(r,A)$ is a trivial fibration, and so in particular a weak equivalence. 

II. Turning to limits, everything is a formal consequence; given a cofibrant weight $G\colon\cd\to\cv$ and a diagram $S\colon\cd\to\Int\cm$, the weighted limit $\{G,S\}$ in \cm is fibrant, but need not be cofibrant. The cofibrant replacement $Q\{G,S\}$ of $\{G,S\}$ gives the required homotopy limit $\{G,S\}_h$. 
\end{proof}
 
\begin{rem}
The homotopy limits and colimits in $\Int\cm$ constructed in the proof have various special properties that need not hold for homotopy limits and colimits in general. For example, the pointwise weak equivalences in the definition of homotopy colimits and limits are actually pointwise trivial fibrations. 
Moreover, assumimg that the model category $\cm$ is functorial, the construction of homotopy limits and colimits in $\Int\cm$ is functorial, which need 
not be the case in a general fibrant \cv-category.   We do, however, have the following weak version of functoriality.
\end{rem}

\begin{propo}\label{prop:functoriality}
  Let $f\colon G\to H$ be a morphism between cofibrant weights in $[\cd\op,\cv]$, and consider a diagram  $S\colon\cd\to\ck$ in the fibrant \cv-category \ck for which the homotopy colimits $G*_h S$ and $H*_h S$ exist. Then there is a morphism $f*_h S\colon G*_h S\to H*_h S$ for which the diagram
$$\xymatrix{
G \ar[r]^-{\delta_G} \ar[d]_{f} & \ck(S,G*_h S) \ar[d]^{\ck(S,f*_h S)} \\
H \ar[r]_-{\delta_H} & \ck(S,H*_h S) }$$
in $[\cd\op,\cv]$ commutes up to homotopy.
\end{propo}

\proof
In the solid part of the diagram
$$\xymatrix{
\ck(G*_h S,H*_h S) \ar[r]^-{\beta_G} & [\cd\op,\cv](G,\ck(J,H*_h S)) \\
I \ar@{.>}[u]^{f*_h S} \ar[r]_-{\delta_H} & [\cd\op,\cv](H,\ck(J,H*_h S)) \ar[u]_{f^*} }$$
$\beta_G$ is (a component of) the weak equivalence in \cv defining the homotopy colimit $G*_h S$, and $\delta_H$ is the counit of the homotopy colimit $H*_h S$, while $f^*$ is given by composition with $f$. Since $\beta_G$ is a weak equivalence between fibrant objects, and $I$ is cofibrant, there is a factorization $f*_h S$ up to homotopy. By naturality,  the composite of $\beta_G$ and $f*_h S$ is the map $I\to[\cd\op,\cv](G,\ck(J,H*_h S))$ corresponding to the composite 
$$\xymatrix{
G \ar[r]^-{\delta_G} & \ck(S,G*_h S) \ar[rr]^-{\ck(S,f*_h S)} && \ck(S,H*_h S) }$$
whence, in view of Proposition~\ref{prop:two-notions}, the result follows. 
\endproof

 Recall that a {\em conical colimit} in a \cv-category \ck is a colimit $\Delta I*S$ where $S\colon\cj\to\ck$ is a diagram defined on the free \cv-category on an ordinary category and $\Delta I$ is the weight which is constant at the unit object $I$.
  
If $\ck$ is fibrant then we may define the {\em homotopy colimit} $\hocolim S$ of $S$ as the homotopy colimit $K*S$ for a cofibrant replacement $K$ of $\Delta I$; by Proposition~\ref{prop:hcolimit-uniqueness} this is independent of the cofibrant replacement. We have a cocone
$$
\delta\colon K\to\ck(S,\hocolim S).
$$
and the natural transformation $q\colon K\to\Delta I$ induces the comparison morphism
$$
k\colon\hocolim S\to\colim S
$$
 (provided, of course, that both colimits exist). 
Observe that the hom-objects of a $\cv$-category $\cj$ are coproducts of $I$, and so are cofibrant in $\cv$. {\em Homotopy limits} in $\ck$ are de\-fi\-ned as homotopy colimits in $\ck\op$. 

In model $\cv$-categories, our homotopy colimits are weakly equivalent to standard ones provided that the diagram $S$ is objectwise cofibrant \cite[Proposition~1]{V}.  
If $\cm$ is a model $\cv$-category and $S\colon\cj\to\Int\cm$ then $\hocolim S$ is a fibrant replacement of $K\ast S$. This definition 
is classical for $\cv=\SSet$ where $K=B(-,\downarrow\cj)\op$: see \cite{H}.

\begin{propo}\label{prop:hty-filtered}
Let $\cv$ be a $\lambda$-combinatorial monoidal model category satisfying the monoid axiom, $\cm$ a $\lambda$-com\-bi\-na\-to\-rial model $\cv$-category, and \cj  a $\lambda$-filtered ordinary category. Let $\Delta I\colon \overline{\cj}\op\to \cv$ be the functor constant at the unit object $I$, and $q\colon K\to \Delta I$ its cofibrant replacement. Then $q*S\colon K*S\to \Delta I*S$ is a weak equivalence for any $S\colon \overline{\cj}\to\cm$.
\end{propo}

\proof
I. First consider the case where \cj has a terminal object $t$. Then $\Delta I\colon\overline\cj\op\to\cv$ is the representable \cv-functor $\overline\cj(-,t)$, and so is cofibrant. Then $q\colon K\to\Delta I$ is a weak equivalence between fibrant and cofibrant objects, and so is a homotopy equivalence. It follows that $q*S$ is a weak equivalence.

II. Since $\lambda$-presentable objects in $\cv$ form a strong generator and $\cm$ has powers by all those objects, conical colimits in $\cm$ may be calculated in $\cm_0$. An arbitrary $\lambda$-filtered category $\cj$ may be written as the union $\cup_{h\in H} \cj_h$ of a $\lambda$-filtered set  of small subcategories $\cj_h$ with a terminal object: write $J_h\colon\cj_{h}\to\cj$ for the inclusions. Each $\cv$-functor 
$$
[\overline{J}_h\op,\cv]\colon [\overline{\cj}\op,\cv]\to[\overline{\cj}_h\op,\cv]
$$
has a left adjoint $\Phi_h$ given by left Kan extension.

Consider a diagram $S\colon\overline\cj\to\cm$ and let $S_h\colon \overline\cj_h\to\cm$ be its restriction for each $h\in H$. Then
$$
\Delta I = \colim_{h\in H}\Phi_h(\Delta I_h)
$$
and similarly $S$ is a colimit of the left Kan extensions of the $S_h$. 

Since $\cv$ is $\lambda$-combinatorial, all presheaf categories $[\overline{\cj}\op,\cv]$ and $[\overline{\cj}_h\op,\cv]$ are $\lambda$-combinatorial too, and so cofibrant replacement functors on them preserve $\lambda$-filtered colimits: see Remark~\ref{rmk:combinatorial}. % \cite[3.1]{R}. 
Thus we have $K=\colim_h K_h$ where $K_h$ is the cofibrant replacement of $(\Delta I)\overline{J}_h^{\op}$.

The canonical maps $q_h*S_h\colon K_h*S_h\to \Delta I_h*S_h$ are weak equivalences by part~I of the proof. Since \cm  is $\lambda$-combinatorial, we know that $\lambda$-filtered colimits of weak equivalences are weak equivalences \cite{ D,R}; thus the induced map $\colim_{h\in H} K_h*S_h\to \Delta I_h*S_h$ is a weak equivalence; but this map is just $q*S\colon K*S\to\Delta I*S$. 
\endproof

  \begin{theo}\label{thm:hocolim-colim}
Let $\cv$ be a $\lambda$-combinatorial monoidal model category satisfying the monoid axiom, $\cm$ a $\lambda$-com\-bi\-na\-to\-rial model $\cv$-category, and \cj a $\lambda$-filtered ordinary category.
If $S\colon\overline\cj\to\cm$ lands in $\Int\cm$ then the canonical comparison  $k\colon \hocolim S\to \colim S$ is a weak equivalence. 
  \end{theo}

  \begin{proof}
Since $S$ has fibrant values, $\colim S$ is fibrant by Remark~\ref{rmk:combinatorial}; thus since 
$r\colon K*S\to R(K*S)$ is a trivial cofibration, the 
map $q*S\colon K*S\to \Delta I*S$ extends along $r$ to give a map $k\colon R(K*S)\to \Delta I*S$. Now $q*S$ is a weak equivalence by Proposition~\ref{prop:hty-filtered}, and so $k$ too is a weak equivalence. 
  \end{proof}
 
\begin{rem}\label{rmk:flexible}
For $\cv=\Cat$, cofibrant weights are precisely flexible weights: see \cite{La}. Since colimits weighted by flexible weights are bicolimits
(see \cite{BKPS} and \cite{La1}), a consequence of Theorem~\ref{thm:hocolim-colim} is that filtered colimits in $\Cat$ are bicolimits: see \cite[5.4.9]{MP}.
We are indebted to J. Bourke for this observation (see \cite{Bo} as well).

This relies on the fact that in $\Cat$ every weak equivalence is an equivalence. For a general combinatorial model 2-category this need not be the case, and so filtered colimits need not be bicolimits; indeed it need not even be the case in a presheaf 2-category.
\end{rem}

We also have the following more general result. 

\begin{propo}
Let $\cv$ be a $\lambda$-combinatorial monoidal model category satisfying the monoid axiom, $\cm$ a $\lambda$-com\-bi\-na\-to\-rial model $\cv$-category, and \cj a $\lambda$-filtered ordinary category.
If $S\colon\overline\cj\to\cm$ has cofibrant values, then $\hocolim RS$ is weakly equivalent to $\colim S$; if $S$ lands in $\Int\cm$, then $\hocolim S$ is weakly equivalent to $\colim S$, in the sense that they are isomorphic in the homotopy category. 
\end{propo}
\begin{proof}
Write $\Delta I\colon\overline{\cj}\op\to\cv$ for the $\cv$-functor which is constant at the identity, and $q\colon K\to\Delta I$ for its cofibrant replacement. 
Then $q*S\colon \hbox{$K*S$}\to(\Delta I)*S$ is a weak equivalence by Proposition~\ref{prop:hty-filtered}, and of course $r\colon K*S\to R(K*S)$ is a weak equivalence. Since, by Theorem~\ref{thm:colimit-existence}, we may construct $\hocolim RS$ as $R(K*RS)$, it will suffice to show that $R(K*r)\colon R(K*S)\to R(K*RS)$ is a weak equivalence. But this will certainly be the case if $R(K*r)$ is a homotopy equivalence, and so it will suffice to show that $\cm(R(K*r),M)\colon\cm(R(K*RS),M)\to\cm(R(K*S),M)$ is a weak equivalence for all $M\in\Int\cm$, and finally this will be true if and only if $\cm(K*r,M)\colon\cm(K*RS,M)\to\cm(K*S,M)$ is a weak equivalence for all $M\in\Int\cm$. By the universal property of the colimits $K*S$ and $K*RS$, this is equivalent to 
$$\xymatrix @C4pc {
[\overline\cj\op,\cv](K,\cm(RS,M)) \ar[r]^{\cm(r,M)_*} & 
[\overline\cj\op,\cv](K,\cm(S,M)) }$$
being a weak equivalence. Now $r\colon S\to RS$ is a trivial cofibration and $M$ fibrant, so $\cm(r,M)$ is a weak equivalence between fibrant objects; thus the displayed map is a weak equivalence because $K$ is cofibrant.
\end{proof}

\section{Preservation of homotopy colimits}\label{sect:pres}

\begin{defi}\label{defn:preservation}  
Let $F\colon\ck\to\cl$ be a $\cv$-functor between fibrant $\cv$-categories, $G\colon\cd\op\to\cv$ a cofibrant weight, and $S\colon\cd\to\ck$ a diagram. 
We say that $F$ {\em preserves} the homotopy weighted colimit $G\ast_h S$ when the composite
$$
G \xrightarrow{\quad  \delta\quad} \ck(S,G\ast_h S) \xrightarrow{\quad F\quad} \cl(FS,F(G\ast_h S)).
$$
exhibits $F(G\ast_hS)$ as the homotopy colimit $G\ast_h FS$.
\end{defi}

\begin{rem}\label{rmk:preservation}
(1) Provided that the homotopy colimit $G*_h FS$ exists, the composite above induces a morphism
$$
l\colon G\ast_h FS\to F(G\ast_h S)
$$
and $F$ preserves $G\ast_hS$ if and only if this morphism is a homotopy equivalence.

(2) In particular, given a diagram $S\colon\overline{\cj}\to\ck$ for an ordinary category $\cj$, we say that $F$ preserves the homotopy colimit of $S$ when
$$
\Delta I \xrightarrow{\quad  \delta\quad} \ck(S,\hocolim S) \xrightarrow{\quad F \quad} \cl(FS,F\hocolim S).
$$
exhibits $F\hocolim S$ as a homotopy colimit of $FS$.
\end{rem}

We now want to show that representable functors preserve homotopy limits, as well as considering the extent to which they preserve homotopy colimits. Usually, a representable functor has codomain \cv, but we are only considering homotopy limits or colimits in fibrant \cv-categories, and \cv need not be fibrant. If \ck is a fibrant \cv-category, then the representable $\ck(K,-)$ will take values in the full subcategory of fibrant objects in \cv, but in general this subcategory need not be fibrant. If, however, we suppose that all objects of \cv are cofibrant, then the full subcategory of fibrant objects is $\Int\cv$, and this finally is a fibrant \cv-category. 

\begin{lemma}\label{lemma:reps-cts}
Let $\cv$ have all objects cofibrant and let $\ck$ be a fibrant $\cv$-category. Then $\ck(A,-)\colon\ck\to\Int\cv$ preserves weighted homotopy limits 
for each object $A$ in $\ck$.
\end{lemma}
\begin{proof}
Consider a cofibrant weight $G\colon\cd\to\cv$, a diagram $S\colon\cd\to\ck$, a 
natural transformation 
$$
\delta\colon G\to\ck(L,S)
$$
exhibiting $L$ as the homotopy limit $\{G,S\}_h$, as well as the corresponding natural transformation
$$
\beta\colon \ck(-,L)\to[\cd,\cv](G,\ck(-,S))
$$
whose components are weak equivalences. Form the composite $\delta'$ given by 
$$\xymatrix @C3pc {
G \ar[r]^-{\delta} & \ck(L,D) \ar[r]^-{\ck(A,-)} & \cv(\ck(A,L),\ck(A,S)) }$$
and  the corresponding natural transformation
$$
\beta'\colon\cv(-,\ck(A,L))\to[\cd,\cv](G,\cv(-,\ck(A,S))) .
$$
We have to show that $\beta'$ is a pointwise weak equivalence. The components $\beta_A\colon\ck(A,L)\to[\cd,\cv](G,\ck(A,S))$ of $\beta$ are weak equivalences between fibrant objects, and so 
$$\cv(X,\beta_K)\colon \cv(X,\ck(A,L))\to \cv(X,[\cd,\cv](G,\ck(A,S)))$$
is also a weak equivalence for each $X\in\cv$. But the composite of the weak equivalence $\cv(X,\beta_K)$ with the canonical isomorphism 
$$\cv(X,[\cd,\cv](G,\ck(A,S)))\cong[\cd,\cv](G,\cv(X,\ck(A,S)))$$ 
is the $X$-component $\beta'_X$ of $\beta'$. Thus $\beta'$ is indeed a pointwise weak equivalence.
\end{proof}

\begin{defi}\label{defn:presentable}
Let \cv have all objects cofibrant and let \ck be a fibrant \cv-category. An object $A$ of \ck is said to be  {\em homotopy} $\lambda$-{\em presentable} 
when $\ck(A,-)\colon\ck\to\Int\cv$ preserves homotopy $\lambda$-filtered colimits.
\end{defi}

\begin{lemma}\label{pres1}
Let $F,G\colon \ck\to\cl$ be \cv-functors between fibrant \cv-ca\-te\-go\-ries, and let $\phi\colon F\to G$ be a pointwise homotopy equivalence. Then $F$ preserves a homotopy colimit in \ck if and only if $G$ does so.
\end{lemma}

\proof
Let $H\colon\cd\op\to\cv$ be a cofibrant weight, and $S\colon\cd\to\ck$ a diagram in \ck, and let 
$$\xymatrix{
H \ar[r]^-{\eta} & \ck(S,H*_h S) }$$
exhibit $H*_h S$ as the homotopy colimit in \ck. 

In the commutative diagram 
$$\xymatrix{
\cl(G(H*_hS),L) \ar[r]^{\cl(\phi,1)} \ar[d]_{\cl(GS,-)} & \cl(F(H*_hS),L) \ar[d]^{\cl(FS,-)} \\
\dhom(\cl(GS,G(H*_hS)),\cl(GS,L)) \ar[d]_{\dhom(G,1)} & \dhom(\cl(FS,F(H*_hS)),\cl(FS,L)) \ar[d]^{\dhom(F,1)} \\
\dhom(\cl(S,H*_HS),\cl(GS,L)) \ar[d]_{\dhom(\eta,1)} \ar[r]_{\dhom(1,\cl(\phi,1))} & \dhom(\cl(S,H*_hS),\cl(FS,L)) \ar[d]^{\dhom(\eta,1)} \\
\dhom(H,\cl(GS,L)) \ar[r]_{\dhom(1,\cl(\phi,1))} & \dhom(H,\cl(FS,L)) 
}$$
where we write $\dhom$ for $[\cd\op,\cv]$, 
the homotopy colimit $H*_hS$ is preserved by $F$ if and only if the right vertical composite is a weak equivalence (in \cv), and is preserved by $G$ if and ony if the left vertical composite is a weak equivalence. But the top and bottom horizontal maps are both weak equivalences, since $\phi$ is a pointwise homotopy equivalence, so these conditions are equivalent. 
\endproof

\begin{lemma}\label{pres2}
Let $F\colon \ck\to\cl$ and $G\colon\cl\to\cm$ be \cv-functors between fibrant \cv-categories. If $F$ preserves a homotopy colimit $H*_hS$,  then $G$ preserves $H*_hFS$ if and only if $GF$ preserves $H*_hS$.
\end{lemma}

\proof
Given that $F$ preserves $H*_hS$, then for $GF$ to preserve $H*_hS$ is literally the same thing as for $G$ to preserve $F(H*_hS)$. 
\endproof

\begin{lemma}\label{lemma:reflect}
  Let $(F_\kappa\colon\ck\to\cl_\kappa)_{\kappa\in K}$ be a family of \cv-functors 
between fibrant \cv-categories, and suppose that they jointly reflect homotopy equivalences. Then the $F_\kappa$ jointly reflect any type of homotopy colimit which they preserve. 
\end{lemma}

\proof
Let $S\colon\cd\to\ck$ be a diagram in \ck and $G\colon\cd\op\to\cv$ a cofibrant weight. Suppose that $\delta\colon G\to\ck(S,G*_h S)$ exhibits $G*_h S$ as the homotopy colimit, and that this is preserved by the $F_\kappa$. Now let $C$ be an object of \ck, and $\gamma\colon G\to \ck(S,C)$ a morphism in $[\cd\op,\cv]$. By the universal property of $G*_h S$, there is a morphism $w\colon G*_h S\to C$ in \ck making the diagram 
$$\xymatrix{
I \ar[d]_{\gamma} \ar[r]^{w} & \ck(G*_h S,C) \ar[d]^{\ck(S,-)} \\
[\cd\op,\cv](G,\ck(S,C)) & [\cd\op,\cv](\ck(S,G*_h S),\ck(S,C)) \ar[l]^-{\delta^*} }$$
commute up to homotopy. 

Now $\gamma$ exhibits $C$ as the homotopy colimit of $S$ weighted by $G$ if and only if $w\colon G*_h S\to C$  is a homotopy equivalence in $\ck$. If the $F_\kappa$ preserve the homotopy colimit $G*_h S$ and each 
$$\xymatrix{
G \ar[r]^-{\gamma} & \ck(S,C) \ar[r]^-{F_\kappa} & \cl_\kappa(F_\kappa S,F_\kappa C) }$$ 
exhibits $F_\kappa C$ as $G*_h F_\kappa S$ then each $F_\kappa w$ is a homotopy equivalence in $\cl_\kappa$. Thus if the $F_\kappa$ jointly reflect homotopy equivalences then $w$ is a homotopy equivalence, and so $\gamma$ exhibits $C$ as the homotopy colimit of $S$ weighted by $G$.
\endproof

\begin{lemma}\label{lemma:qe}
Let \cv be a $\lambda$-combinatorial monoidal model category, let \ca be a small \cv-category, and let \ck be a fibrant \cv-category with $\lambda$-filtered homotopy colimits. Let $F\colon\ca\to\ck$ be a \cv-functor whose values are homotopy $\lambda$-presentable, and let $E\colon\ck\to[\ca\op,\cv]$ be the \cv-functor sending $X\in\ck$ to $\ck(F-,X)$. Then the composite $Q\circ E\colon\ck\to\Int[\ca\op,\cv]$ preserves $\lambda$-filtered homotopy colimits. 
\end{lemma}

\proof
By Remark~\ref{rmk:preservation}, we have a comparison 
$$\ell\colon \hocolim QES\to QE(\hocolim S)$$ 
for each $\lambda$-filtered diagram $S$ in \ck, and we are to show that this is a homotopy equivalence in $\Int[\ca\op,\cv]$; equivalently, a weak equivalence. 

There is an  evaluation functor $\ev_A\colon\Int[\ca\op,\cv]\to\Int\cv$ for each object $A\in\ca$. These preserve homotopy colimits by the construction of \ref{thm:colimit-existence}, and they jointly reflect weak equivalences, thus they jointly reflect homotopy colimits. But the composites $\ev_A QE$ preserve homotopy $\lambda$-filtered colimits because the $A$ are homotopy $\lambda$-presentable in \ck.
\endproof

\begin{propo}\label{prop:preservation}
  Let \cm and \cn be $\lambda$-combinatorial model \cv-categories, and let $F\colon\cm\to\cn$ be a \cv-functor which preserves $\lambda$-filtered colimits, preserves weak equivalences between fibrant objects, and which maps $\Int\cm$ to $\Int\cn$.  Then the induced $F\colon\Int\cm\to\Int\cn$  preserves $\lambda$-filtered homotopy colimits. 
\end{propo}

\proof
Let \cj be the free \cv-category on a $\lambda$-filtered ordinary category, let $\Delta I\colon\cj\op\to\cv$ be the weight for conical colimits, and let $q\colon K\to \Delta I$ be the cofibrant replacement of $\Delta I$ in $[\cj\op,\cv]$. Let $S\colon\cj\to\cm$ take values in $\Int\cm$. Thus $\Delta I*S$ is a synonym for $\colim S$. 

Then $\hocolim FS$ may be constructed as the fibrant replacement $R(K*FS)$ of $K*FS$, and similarly $F\hocolim S$ as $FR(K*S)$. The fibrant replacement map $r\colon K*FS\to R(K*FS)$ is of course a weak equivalence, and if we can show that the composite 
$$\xymatrix{
K*FS \ar[r]^-{\phi} & F(K*S) \ar[r]^-{Fr} & FR(K*S) }$$
is a weak equivalence, where $\phi$ is the canonical comparison map, then  preservation of  $\hocolim S$ will follow. 

To do this, consider the diagram 
$$\xymatrix{
K*FS \ar[d]_{q*FS} \ar[r]^-{\phi} & F(K*S) \ar[r]^-{Fr} \ar[d]_{F(q*S)} & FR(K*S) \ar[d]^{FR(q*S)} \\
\Delta I*FS \ar[r]_-{\phi} & F(\Delta I*S) \ar[r]_-{Fr} & FR(\Delta I*S) }$$
where the lower $\phi$ is once again a canonical comparison map; this time invertible, since $F$ preserves $\colim S$ by assumption. Now $q*S$ and $q*FS$ are weak equivalences by Proposition~\ref{prop:hty-filtered}, and so also $FR(q*S)$ is a weak equivalence since $F$ preserves weak equivalences between fibrant objects. Thus it will suffice to show that $Fr\colon F(\Delta I*S)\to FR(\Delta I*S)$ is a weak equivalence. But $\Delta I*S$ is a $\lambda$-filtered colimit of fibrant objects and so fibrant by Remark~\ref{rmk:combinatorial}, thus $r\colon \Delta I*S\to R(\Delta I*S)$ is a weak equivalence between fibrant objects, and so preserved by $F$. 
\endproof

We close this section with a discussion of limit-colimit commutativity in the homotopical context. 

Let $G\colon\cd\to\cv$ and $H\colon \cc\op\to\cv$ be weights. One says that $G$-weighted limits commute with $H$-weighted colimits, in a \cv-category \cm in which these limits and colimits exist, if the \cv-functor $H*-$ below 
$$\xymatrix{
[\cc,\cm] \ar[r]^-{H*-} & \cm & [\cd,\cm] \ar[r]^-{\{G,-\}} & \cm }$$
preserves $G$-weighted limits; or, equivalently, if the \cv-functor $\{G,-\}$ preserves $H$-weighted colimits. Each of these is in turn equivalent to the condition that, for each \cv-functor $S\colon\cd\ox\cc\to\cm$, the canonical comparison $K*\{G,S\}\to \{G,K*S\}$ is invertible. 

In the homotopy setting things are more delicate, since, as observed above, homotopy limits and homotopy colimits need not be functorial in general. For this reason, we restrict to the case of homotopy limits and colimits in $\Int\cm$, for a model \cv-category \cm. We then say that homotopy $G$-weighted limits commute with homotopy $H$-weighted colimits, if the composite of the canonical maps 
$$\xymatrix{
H*Q\{G,S\} \ar[r]^{H*q} & H*\{G,S\} \ar[r] & 
\{G,H*S\} \ar[r]^-{\{G,r\}} & \{G,R(H*S)\} }$$
is a weak equivalence, for any $S\colon\cd\ox\cc\to\cm$ taking values in $\Int\cm$. Here the fibrant replacement of the domain $H*Q\{G,S\}$ has the form $H*_h\{G,S\}_h$, while the cofibrant replacement of the codomain \hbox{$\{G,R(H*S)\}$} has the form $\{G,H*_h S\}_h$. 

\begin{propo}\label{prop:filtered-finite}
Let $\cv$ be a $\lambda$-combinatorial monoidal model category in which all objects are cofibrant, and let \cm be a $\lambda$-combinatorial model \cv-category. Then $\lambda$-filtered homotopy colimits commute in  $\Int\cm$ with $\lambda$-presentable homotopy limits.
\end{propo}
\proof
Let \cj be a $\lambda$-filtered ordinary category, and consider  $\Delta I\colon\overline\cj\op\to\cv$ and its cofibrant replacement $q\colon K\to\Delta I$. 
 Let $G\colon\cd\to\cv$ be $\lambda$-presentable and cofibrant, and let $S\colon\cd\ox\overline\cj\to\cm$ take values in $\Int\cm$. 

Consider the commutative diagram 
$$\xymatrix{
K*Q\{G,S\} \ar[r]^{K*q} \ar[d]_{q*Q\{G,S\}} & 
K*\{G,S\} \ar[r]^{\phi_K} \ar[d]_{q*\{G,S\}} & 
\{G,K*S\} \ar[r]^{\{G,r\}} \ar[d]^{\{G,q*S\}} & 
\{G,R(K*S)\} \ar[d]^{\{G,R(q*S)\}} \\
\Delta I*Q\{G,S\} \ar[r]_{\Delta I*q} & 
\Delta I*\{G,S\} \ar[r]_{\phi_{\Delta I}}  & 
\{G,\Delta I*S\} \ar[r]_{\{G,r\}} & 
\{G,R(\Delta I*S)\} 
}$$
in which the maps $\phi_K$ and $\phi_{\Delta I}$ are the canonical comparisons from the non-homotopy situation.

We are to prove that the upper horizontal composite is a weak equivalence. Now $\Delta I*q$ is a $\lambda$-filtered colimit of weak equivalences, so is a weak equivalence by Remark~\ref{rmk:combinatorial}; while $\phi_{\Delta I}$ is invertible, since $\lambda$-presentable limits commute with $\lambda$-filtered colimits. Also $S(D,J)$ is fibrant for all values $D\in\cd$ and $J\in\cj$, thus the values of $\Delta I*S$ are $\lambda$-filtered colimits of fibrant objects, and so fibrant by Remark~\ref{rmk:combinatorial} once again. In other words, $\Delta I*S$ is fibrant in $[\cd,\cm]$.  Thus $r\colon \Delta I*S\to R(\Delta I*S)$ is a weak equivalence between fibrant objects. Since $G$ is cofibrant and \cm is a model \cv-category, $\{G,r\}$ is also a weak equivalence by Ken Brown's lemma. 
This now proves that the lower horizontal composite is a weak equivalence. 

It will suffice, therefore, to prove that the left and right vertical maps are weak equivalences. The case of the left vertical map $q*Q\{G,S\}$ follows from Proposition~\ref{prop:hty-filtered}. By the same proposition, we know that $q*S$ is a weak equivalence, and so that $R(q*S)$ is a weak equivalence between fibrant objects; thus, using Ken Brown's lemma and the fact that $G$ is cofibrant once again, $\{G,R(q*S)\}$ is a weak equivalence.
\endproof

\section{Dwyer-Kan equivalences}\label{sect:DK}

Recall that a \cv-functor $W\colon\ck\to\cl$ is said to have a property ``locally'', if each of the induced maps $\ck(K,K')\to\cl(WK,WK')$ between hom-objects has the property (in \cv). 
In particular, $W\colon\ck\to\cl$ is {\em locally a weak equivalence} if each $W\colon\ck(K,K')\to\cl(WK,WK')$ is a weak equivalence in \cv. 

Recall also that a $\cv$-functor $W\colon\ck\to\cl$ is called a {\em Dwyer-Kan equivalence} or just a {\em weak equivalence} if
\begin{enumerate}
\item it is {\em locally a weak equivalence}, in the above sense, and %  that the induced morphisms $\hom(K_1,K_2)\to \hom(W(K_1),W(K_2))$ are weak equivalences for all objects $K_1$
% and $K_2$ of $\ck$ and
\item it is {\em homotopically surjective on objects}, in the sense that for each object $L\in\cl$ there is an object $K\in\ck$ and a homotopy equivalence $L\to WK$.
%of $\cl$ is homotopy equivalent to $W(K)$ for some object $K$ of $\ck$.
\end{enumerate}
If \cl is fibrant, then there exists a homotopy equivalence $L\to WK$ if and only if there exists a homotopy equivalence $WK\to L$: see Remark~\ref{rmk:hty-equiv-symmetry}.

For well-behaved \cv, these Dwyer-Kan equivalences are the weak equivalences for a model structure on the category of small \cv-categories: see \cite{BM,L,M}.
 
\begin{propo}\label{hococomplete}
  Let \cv be a combinatorial monoidal model category, and let $W\colon\ck\to\cl$ be a Dwyer-Kan equivalence between fibrant \cv-categories. Then
  \begin{enumerate}
  \item $W$ preserves any existing homotopy colimits;
    \item $W$ creates homotopy colimits, in the sense that if $S\colon\cd\to\ck$ is a diagram for which the homotopy colimit $G*_h WS$ exists in \cl, then the homotopy colimit $G*_h S$ exists in \ck (and is preserved by $W$);
      \item $W$ preserves and reflects presentability, in the sense that if \cl has homotopy $\lambda$-filtered colimits, then an object $A\in\ck$ is homotopy 
      $\lambda$-presentable if and only if  $WA\in\cl$ is so. 
  \end{enumerate}
\end{propo}

\proof
(1) Let $G\colon\cd\op\to\cv$ be a cofibrant weight, $S\colon\cd\to\ck$ a diagram in \ck, and $\delta\colon G\to \ck(S,K)$ a morphism in $[\cd\op,\cv]$ exhibiting $K$ as the homotopy weighted colimit $G*_h S$. As earlier, we write $\dhom$ for $[\cd\op,\cv]$. In the diagram 
$$\xymatrix{
\ck(K,A) \ar[dd]_-{\ck(S,-)} \ar[r]^-{W_{K,A}} & \cl(WK,WA) \ar[d]^{\cl(WS,-)} \\
& \dhom(\cl(WS,WK),\cl(WS,WA)) \ar[d]^-{W^*} \\
\dhom(\ck(S,K), \ck(S,A)) \ar[d]_-{\delta^*} \ar[r]^-{W_*} & 
\dhom(\ck(S,K),\cl(WS,WA)) \ar[d]^-{\delta^*} \\
\dhom(G,\ck(S,A)) \ar[r]_-{W_*} & \dhom(G,\cl(WS,WA)) }$$
the left vertical $\delta^* \circ \ck(S,-)$ is a weak equivalence by the universal property of  $K=G*_h S$; while the lower horizontal  $W_*$ is a weak equivalence because $G$ is cofibrant and $W\colon\ck(S,A)\to\cl(WS,WA)$ is a weak equivalence between fibrant objects.  Since the upper horizontal  $W_{K,A}$ is a weak equivalence, it follows that the right vertical $\delta^* \circ W^*\circ\cl(WS,-)$ is one too. But this implies, using the fact that any $B\in\cl$ is homotopy equivalent to $WA$ for some $A\in\ck$,  that the composite 
$$\xymatrix{ 
G \ar[r]^-{\delta} & \ck(S,K) \ar[r]^-{W} & \cl(WS,WK) }$$
exhibits $WK$ as $G*_h WS$. 

\noindent (2)
Let $G\colon\cd\op\to\cv$ be a cofibrant weight, $S\colon\cd\to\ck$ a diagram in \ck, and $\gamma\colon G\to \cl(WS,L)$ a morphism in $[\cd\op,\cv]$ exhibiting $L$ as the homotopy weighted colimit $G*_h WS$. Since $W$ is homotopy surjective, we may suppose without loss of generality that the object $L$ has the form $WK$ for some $K\in\ck$. 

Since $W_{S,K}\colon\ck(S,K)\to\cl(WS,WK)$ is a weak equivalence between fibrant objects, and $G$ is (fibrant and) cofibrant, there is a map $\delta$ making the triangle
$$\xymatrix{
& \ck(S,K) \ar[d]^{W_{S,K}} \\
G \ar[ur]^{\delta} \ar[r]_-{\gamma} & \cl(WS,WK) }$$
commute up to homotopy. By Proposition~\ref{prop:hcolimit-uniqueness}, the composite $W_{S,K}\delta$ also exhibits $WK$ as the homotopy colimit $G*_h WS$; henceforth we take this to be $\gamma$, so that the triangle commutes on the nose. 

Consider once again the large diagram appearing in the proof of (1). This time we know that the composite $\delta^*\circ W^*\circ \cl(WS,-)$ is a weak equivalence; since of course $W_{K,A}$ is also a weak equivalence, the upper composite is a weak equivalence hence so too is the lower one. The lower horizontal $W_*$ is a weak equivalence just as before, hence $\delta^*\circ\ck(S,-)$ is a weak equivalence, as required. 

(3) Suppose that \cl has homotopy $\lambda$-filtered colimits; it follows by (2) that \ck does so too, and that $W\colon \ck\to\cl$ preserves them. Let $A\in\ck$; we are to show that $\ck(A,-)$ preserves homotopy $\lambda$-filtered colimits if and only if $\cl(WA,-)$ does so. But $W$ preserves homotopy $\lambda$-filtered colimits by (1), thus by Lemma~\ref{pres2} $\cl(WA,-)$ preserves them if and only if $\cl(WA,W-)$ does; while $W_{A,-}\colon \ck(A,-)\to\cl(WA,W-)$ is a pointwise homotopy equivalence, thus by Lemma~\ref{pres1} $\cl(WA,W-)$ preserves homotopy $\lambda$-filtered colimits if and only if $\ck(A,-)$ does so. 
\endproof

The next result holds without the \cv-categories needing to be fibrant.

Let $V\colon\ca\to\cb$ be a \cv-functor between small \cv-categories. Then composition with $V$ induces a \cv-functor $V^*\colon[\cb\op,\cv]\to[\ca\op,\cv]$, and this has a left adjoint $V_!\dashv V^*$ given by left Kan extension. If we give the presheaf categories the projective model structure, then $V^*$ preserves fibrations and weak equivalence, more or less by definition; thus $V_!\dashv V^*$ becomes a Quillen adjunction.

\begin{propo}\label{prop:DKimpliesQ}
  Let $V\colon\ca\to\cb$ be a Dwyer-Kan equivalence between small \cv-categories. 
Then the induced adjunction $V_!\dashv V^*$ is a Quillen equivalence. 
\end{propo}

\proof
Write $n\colon 1\to V^*V_!$ for the unit and $e\colon V_!V^*\to 1$ for the counit of the adjunction.

First we show that,  for any cofibrant $M\colon\ca\op\to\cv$, the composite
$$\xymatrix{
M \ar[r]^-n & V^* V_! M \ar[r]^{V^*r} & V^* R V_! M }$$
is a weak equivalence, where $r\colon 1\to R$ denotes the fibrant replacement functor. Now $V^*$ preserves the weak equivalence $r\colon V_! M\to RV_!M$, so it will suffice to prove that the unit $n\colon M\to V^* V_!M$ of the original adjunction is a weak equivalence; in other words, that for each $A\in\ca$, the map $MA\to\ev_{VA} V_!M$ is a weak equivalence. But this map is obtained by applying $M*-$ to the pointwise weak equivalence 
$$\xymatrix @C5pc {
\ca \ar@/^1pc/[rr]^{\ca(A,-)}_{~}="1" \ar@/_1pc/[rr]_{\cb(VA,V-)}^{~}="2" && \Int\cv 
\ar@{=>}"1";"2"^{V_{A,-}} }$$
and weak equivalences in $\Int\cv$ are in fact homotopy equivalences, so $V_{A,-}\colon\ca(A,-)\to\cb(VA,V-)$ is in fact a pointwise homotopy equivalence. 
Thus $M*V_{A,-}\colon MA\to \ev_{VA}V_!M$ is a weak equivalence by Lemma~\ref{lemma:invariance} below.

This proves that the unit of the derived adjunction is invertible; we now turn to the counit. For this, we should show that the composite 
$$\xymatrix{
V_!QV^*N \ar[r]^{V_!q} & V_!V^*N \ar[r]^{e} & N }$$
is a weak equivalence for any fibrant $N$. In  the diagram
$$\xymatrix{
V^*V_!QV^*N \ar[r]^{V^*V_!q} & V^*V_!V^*N \ar[r]^{V^*e} & V^*N \\
QV^*N \ar[u]^{n} \ar[r]_{q} & V^*N \ar[u]^{n} \ar@{=}[ur] }$$
$q$ is of course a weak equivalence, and $n\colon QV^*N\to V^*V_!QV^*N$ is a weak equivalence by the first half of the proof. Thus  $V^*(e.V_!q)$ is a weak equivalence. But  
it follows easily from the fact that $V$ is homotopically surjective on objects that  $V^*$ reflects weak equivalences; thus $e.V_!q$ is a weak equivalence as required. 
\endproof

The following lemma is closely related to parts of Proposition~\ref{prop:hcolimit-uniqueness}; indeed it could be deduced from that proposition if we restricted to the case of fibrant \cv-categories. But in some sense it is more basic, and so we give an independent proof. 

\begin{lemma}\label{lemma:invariance}
Let \cm be a model \cv-category, and let $S,T\colon\cd\to\cm$ take values in $\Int\cm$. Let $w\colon S\to T$ be a pointwise weak equivalence, and let $G\colon\cd\op\to\cv$ be a cofibrant weight. Then $G*w\colon G*S\to G*T$ is a weak equivalence in $\cm$.
\end{lemma}

\proof
Since $S$ and $T$ land in $\Int\cm$, the map $w$ is in fact a pointwise homotopy equivalence. Thus the induced map $\cm(w,X)\colon\cm(T,X)\to\cm(S,X)$ is a pointwise homotopy equivalence for any $X\in\cm$. 

If in fact $X$ is fibrant, then $\cm(S,X)$ and $\cm(T,X)$ are fibrant in $[\cd\op,\cv]$, and so the induced map 
$$[\cd\op,\cv](G,\cm(T,X))\to [\cd\op,\cv](G,\cm(S,X))$$
is a weak equivalence (since $G$ is cofibrant). Thus in turn the map 
$$\cm(G*w,X)\colon \cm(G*T,X)\to \cm(G*S,X)$$
is a weak equivalence for all fibrant $X$. 

Now $G*S$ and $G*T$ are cofibrant colimits of cofibrant objects, so are cofibrant; but they need not be fibrant. In the commutative diagram
$$\xymatrix @C5pc {
\cm(R(G*T),X) \ar[r]^{\cm(R(G*w),X)} \ar[d]_{\cm(r,X)} & \cm(R(G*S),X) \ar[d]^{\cm(r,X)}\\
\cm(G*T,X) \ar[r]_{\cm(G*w,X)} & \cm(G*S,X) }$$
the vertical maps are weak equivalences (since $r$ is a trivial cofibration and $X$ is fibrant), and $\cm(G*w,X)$ is a weak equivalence, thus also $\cm(R(G*w),X)$ is a weak equivalence. But this is true for all $X\in\Int\cm$, and so in fact $R(G*w)$ is a homotopy equivalence, and in particular a weak equivalence. Finally $R$ reflects weak equivalences, and so $G*w$ is a weak equivalence as required. 
\endproof

\section{Characterization of small homotopy orthogonality classes}\label{sect:orthogonality-classes}

\begin{propo}\label{prop:limits-orthog}
Let $\cv$ be a monoidal model category satisfying the monoid axiom, and $\ck$ a fibrant $\cv$-category. Then homotopy orthogonality classes in $\ck$ are closed under any existing weighted homotopy limits.
\end{propo}
\begin{proof}
It suffices to show that objects homotopy orthogonal to a single morphism $f\colon A\to B$ are closed under existing weighted homotopy limits. Let $S\colon\cd\to\ck$ 
be a diagram with each $SD$ homotopy orthogonal to $f$,  let $G\colon\cd\to\cv$ be a cofibrant weight, and suppose that the homotopy limit $\{G,S\}_h$ 
exists in $\ck$. Then we have a commutative diagram
$$
\xymatrix @C5pc {
\ck(B,\{G,S\}_h) \ar[r]^-{\ck(f,\{G,S\}_h)} \ar[d] &
\ck(A,\{G,S\}_h) \ar[d] \\
[\cd,\cv](G,\ck(B,S)) \ar[r]_-{[\cd,\cv](G,\ck(f,D))} & [\cd,\cv](G,\ck(A,S)) }
$$
in which the vertical maps are weak equivalences, by definition of the homotopy limits. Now $G$ is cofibrant, and $\ck(f,D)$ is a (pointwise)
weak equivalence between fibrant objects, so $[\cd,\cv](G,\ck(f,D))$ is also a weak equivalence. It follows that $\ck(f,\{G,D\}_h)$ is a weak equivalence 
and so that $\{G,D\}_h$ is homotopy orthogonal to $f$. 
\end{proof}

\begin{defi}\label{defn:hty-replete}
Let $\ck$ be a fibrant $\cv$-category. A full subcategory $\ca$ of $\ck$ is called {\em homotopy replete} if an object $D$ in $\ck$
lies in $\ca$ whenever there is a homotopy equivalence $h\colon C\to D$ with $C$ in $\ca$. 
\end{defi}

\begin{defi}\label{defn:accessibly-embedded}
Let $\ck$ be a full subcategory of a category $\cm$. We say that a full subcategory $\ca$ of $\ck$ is $\cm$-{\em accessibly embedded} 
in $\ck$ when there exists a regular cardinal $\lambda$ such that $\ca$ is closed in $\ck$ under all $\lambda$-filtered colimits which 
exist in $\ck$ and are preserved by the inclusion into $\cm$. In this case, we say that $\ck$ is $(\cm,\lambda)$-{\em accessibly embedded} in $\ck$.
\end{defi}

\begin{rem}\label{rmk:fibrants-accessible}
Given a combinatorial model category $\cm$, the full subcategory $\cm\fib$ consisting of fibrant objects is accessibly embedded in $\cm$ by Remark~\ref{rmk:combinatorial}.
In general, this is true neither for the full subcategory $\cm\cof$ nor for $\Int\cm$. Indeed we do not even know that the inclusion of $\Int\cm$  in $\cm$ preserves all existing $\lambda$-filtered colimits for some regular cardinal $\lambda$, although this would be true under Vop\v enka's 
principle. However, there is a regular cardinal $\lambda$ such that each object $X$ in $\Int\cm$ is a $\lambda$-filtered colimit 
$(\delta_d\colon K_d\to X)_{d\in\cd}$ in $\cm$ of objects $K_d$ which are $\lambda$-presentable in $\cm$ and belong to $\Int\cm$.  We now explain why this is the case.

In fact by \cite{D} (or \cite{R}) there is a regular cardinal $\lambda$ such that $\cm$ is locally $\lambda$-presentable, $\cm\fib$ is $\lambda$-accessible, and its inclusion in $\cm$ preserves $\lambda$-filtered colimits and $\lambda$-presentable objects; and, moreover, 
the cofibrant replacement functor $Q\colon\cm\to\cm$ preserves $\lambda$-filtered colimits and $\lambda$-presentable objects. Now, consider 
$X$ in $\Int\cm$ and take its cofibrant replacement $q\colon QX\to X$. Such an  $X$ is a $\lambda$-filtered colimit 
$(\delta_j\colon X_j\to X)_{j\in\cj}$ of objects $X_j\in\cm\fib$ which are $\lambda$-presentable in \cm, and now $QX$ is a $\lambda$-filtered
colimit $(Q\delta_j\colon QX_j\to QX)_{j\in\cj}$ of $\lambda$-presentable objects $QX_j\in\cm$ belonging to $\Int\cm$. Since
$X$ is cofibrant, it is a retract of $QX$. By  the proof of \cite[2.3.11]{MP}, $X$ is a $\lambda$-filtered colimit of objects $QX_j$.

If the object $K\in\Int\cm$ is $\lambda$-presentable in \cm, then any morphism $f\colon K\to X$ factorizes through some $\delta_j$. Now $X$ is a canonical 
$\lambda$-filtered colimit in $\cm$ for the diagram consisting of all $f\colon K\to X$ where $K\in\Int\cm$ is $\lambda$-presentable in $\cm$. But we do not know that 
the objects $K$ are $\lambda$-presentable in $\Int\cm$  because we do not know that $\Int\cm$ is closed in \cm under $\lambda$-filtered colimits (or even that $\Int\cm$ has such colimits). 
\end{rem}

\begin{theo}\label{thm:orthogonality-class}
Let $\cv$ be a tractable monoidal model category and $\cm$ a tractable left proper model $\cv$-category. Then a full subcategory $\ca$ 
of $\Int\cm$ is a small homotopy orthogonality class if and only if it is $\cm$-accessibly embedded, homotopy reflective, and homotopy replete.  
\end{theo}
\begin{proof}
Let $\ca$ be a small homotopy orthogonality class in $\Int\cm$. Then $\ca$ is homotopy reflective by Theorem~\ref{thm:orthog-refl} and homotopy replete by Lemma~\ref{lemma:hty-orthog-invariance}. Moreover by Proposition~\ref{prop:Fhorn} there is a set $\cg$ of morphisms in $\cm$ such that $\ca$ consists of those objects of $\Int\cm$ which are injective in $\cm_0$ with respect to \cg. It follows that \ca is \cm-accessibly embedded in $\Int\cm$.

Conversely, assume that $\ca$ is $\cm$-accessibly embedded, homotopy reflective, and homotopy replete. We shall show that $\ca=\ch$-$\Inj$ 
where $\ch$ consists of all homotopy reflections $\eta_K\colon K\to K^\ast$ with $K$ in $\Int\cm$. Clearly $\ca\subseteq\ch$-$\Inj$. Consider $K$ in $\ch$-$\Inj$. 
Then both $\cm(\eta_K,K)$ and $\cm(\eta_K,K^\ast)$ are weak equivalences, and so by Proposition~\ref{prop:hty-equiv}, also $\eta_K$ is a homotopy equivalence. 
Since $\ca$ is homotopy replete, $K$ belongs to $\ca$.  This proves that $\ca=\ch$-$\Inj$, and if $\ch$ were small, the proof would be complete. 

For each $K\in\Int\cm$, factorize $\eta_K$ as
$$\xymatrix{
K \ar[r]^{g_K} & \tilde{K} \ar[r]^{h_K} & K^* }$$
where $g_K$ is a cofibration and $h_K$ a trivial fibration. Now let $\tilde{\ch}$ consist of the morphisms $g_K\colon K\to\tilde{K}$ (still with  $K\in\Int\cm$). By Lemma~\ref{lemma:hty-orthog-comp}, $\ch$-$\Inj=\tilde{\ch}$-$\Inj$. There is a regular cardinal $\lambda$ 
having the property from Remark~\ref{rmk:fibrants-accessible} and such that $\ca_0$ is $(\cm,\lambda)$-accessibly embedded in $(\Int\cm)_0$. When we speak of an object being $\lambda$-presentable, we shall always mean $\lambda$-presentable in \cm. Let $\cf$ consist of all those cofibrations 
$g_K$ for which $K$ is $\lambda$-presentable. We have $\ca\subseteq\cf$-$\Inj$. Since $\cf$ is a set, it suffices to prove the converse inclusion. 

Assume that $X\in\Int\cm$ belongs to $\cf$-$\Inj$. By  Remark~\ref{rmk:fibrants-accessible}, we know that $X$ is the colimit of the $\lambda$-filtered diagram consisting of all $f\colon K\to X$ with $K\in\Int\cm$ $\lambda$-presentable in \cm. Since $g_K\colon K\to \tilde{K}$ is a cofibration, $\cm(g_K,X)\colon\cm(\tilde{K},X)\to\cm(K,X)$ is a fibration (in $\cv$), and so we can factorize $f$ as $hg_K$, where $h\colon \tilde{K}\to X$.
Since $K$ is $\lambda$-presentable and $\tilde{K}$ is a $\lambda$-filtered colimit of objects of \ca which are $\lambda$-presentable, there is a factorization 
$$\xymatrix{ K \ar[r]^{s} & K' \ar[r]^{t} & \tilde{K} }$$
of $g_K$ for a   $\lambda$-presentable $K'\in\ca$. Now $f=hg_k=hts$, and so $X$ is the ($\lambda$-filtered) colimit of the diagram consisting of the $ht\colon K'\to X$; since the $K'$ are in \ca, so too is $X$.
\end{proof}

\begin{rem}\label{rmk:orthog-class-lambda}
In the situation of Theorem~\ref{thm:orthogonality-class}, let $\ca$ be a small homotopy orthogonality class in $\Int\cm$. By Remark~\ref{rmk:fibrants-accessible}, there is 
a regular cardinal $\lambda$ such that each object $X$ in $\ca$ is a canonical $\lambda$-filtered colimit in $\cm$ of objects belonging
to $\Int\cm$ and $\lambda$-presentable in $\cm$. We can assume that the (cofibration, trivial fibration) factorization in $\cm$ preserves
$\lambda$-filtered colimits and $\lambda$-presentable objects. Then each morphism $f\colon K\to X$ with $K$ in $\Int\cm$ and $\lambda$-presentable 
in $\cm$ factorizes as
$$\xymatrix{ K \ar[r]^{g} & \tilde{K} \ar[r]^{h} & X }$$
where $g$ is a cofibration, $h$ is a trivial fibration, and $\tilde{K}$ is $\lambda$-presentable in $\cm$. Since $\ca$ is homotopy replete,
$\tilde{K}$ belongs to $\ca$; thus $X$ is a canonical $\lambda$-filtered colimit in $\cm$ of objects belonging to $\ca$ and $\lambda$-presentable 
in $\cm$. 
\end{rem}
 
\begin{coro}\label{cor5.7}
Let $\cv$ be a tractable monoidal model category, $\cm$ a tractable left proper model $\cv$-category, and $\ca$ a small homotopy
ortho\-go\-na\-li\-ty class in $\Int\cm$. Then there is a regular cardinal $\lambda$ such that $\ca$ is closed under homotopy $\lambda$-filtered 
colimits in $\Int\cm$.
\end{coro}
\begin{proof}
Let $\cj$ be a generating set of trivial cofibrations in $\cm$.  By Proposition~\ref{prop:Fhorn}, the objects of $\cf$-$\Inj$ are precisly those of $\Hor(\cf)\cap\cj$-$\Inj$. Thus 
$\cf$-$\Inj$ is closed in $\cm$ under 
under $\lambda$-filtered colimits for some regular cardinal $\lambda$. 
Let $q\colon Q(\colim D)\to\colim D$ be a cofibrant replacement. Since $q$ is a trivial fibration and the domains and codomains of morphisms from $\cf$ 
are cofibrant (because they belong to $\Int\cm$), the same argument as in the proof of Lemma~\ref{lemma:hty-orthog-invariance} yields that $Q(\colim D)$ belongs to $\cf$-$\Inj$. 
Let $k\colon\hocolim D\to\colim D$ be a weak equivalence as guaranteed by Theorem~\ref{thm:hocolim-colim}. Since $\hocolim D$ is cofibrant, there is a lifting $f\colon\hocolim D\to Q(\colim D)$ 
with $q f=k$. Thus $f$ is a weak equivalence and, since both $\hocolim D$ and $Q(\colim D)$ belong to $\Int\cm$, also a homotopy equivalence. Hence, by Lemma~\ref{lemma:hty-orthog-invariance}, $\hocolim D$ belongs to $\cf$-$\Inj$. 
\end{proof}

\begin{defi}\label{defn:hty-acc-emb}  
A full subcategory of a fibrant $\cv$-category is called {\em homotopy accessibly embedded} when it is closed under homotopy 
$\lambda$-filtered colimits for some regular cardinal $\lambda$.
\end{defi}

\begin{coro}\label{cor:hty-orthog-characterization}
Let $\cv$ be a tractable monoidal model category and $\cm$ a tractable left proper model $\cv$-category. Then a full subcategory $\ca$ 
of $\Int\cm$ is a small homotopy orthogonality class if and only if it is homotopy accessibly embedded and homotopy reflective.  
\end{coro}
\begin{proof}
The necessity follows from Theorem~\ref{thm:orthogonality-class} and Corollary~\ref{cor5.7}. Since homotopy colimits in $\Int\cm$ are determined up to homotopy equivalence,
a homotopy accessible subcategory of $\Int\cm$ is homotopy replete. Moreover, it is $\cm$-accessibly embedded, thus the sufficiency follows from Theorem~\ref{thm:orthogonality-class}.
\end{proof}

\section{Homotopy locally presentable categories}\label{sect:hlp}

{\em We suppose for the remainder of the paper that \cv is locally present\-able and that every object is cofibrant}. 

Recall \cite{K} that \cv is locally  $\lambda$-presentable as a closed category when it is locally $\lambda$-presentable, and the full subcategory of $\lambda$-presentable objects is closed under tensoring and contains the unit. This assumption allows a good theory of \cv-enriched locally $\lambda$-presentable categories: see \cite{K}. Any symmetric monoidal closed category \cv which is locally presentable as a category is locally $\lambda$-presentable as a closed category for some $\lambda$, and then also for all larger values of $\lambda$: see \cite{vcat}.

\begin{defi}\label{def:lambda-small}
Suppose that \cv is locally $\lambda$-presentable as a closed category. A \cv-category \cx is said to be {\em $\lambda$-small} when it has fewer than $\lambda$ objects and when each hom-object $\cx(x,y)$ is $\lambda$-presentable in \cv.
\end{defi}

\begin{exam}\label{ex:lambda-small}
Suppose that \cj is an ordinary category which is $\lambda$-small in the usual sense that it has fewer than $\lambda$-morphisms. Since $I$ is $\lambda$-presentable, $\cv_0(I,-)\colon\cv_0\to\Set$ preserves $\lambda$-filtered colimits, and thus its left adjoint preserves $\lambda$-presentable objects. Hence each hom-object of $\overline{\cj}$ is $\lambda$-presentable in \cv, and so $\overline{\cj}$ is $\lambda$-small as a \cv-category. 
\end{exam}

\begin{propo}\label{prop:lambda-small-weight}
Suppose that \cv is locally $\lambda$-presentable as a closed category.  If \cd is a $\lambda$-small \cv-category, a weight $G\colon\cd\to\cv$ is $\lambda$-presentable in $[\cd,\cv]$ if and only if $GD$ is $\lambda$-presentable in \cv for all $D\in\cd$. Such a weight $G$ is called {\em $\lambda$-small}.
\end{propo}

\proof
If $G$ is $\lambda$-presentable, then it is a $\lambda$-small colimit of representables. Since the evaluation functors are cocontinuous, $GD$ is a $\lambda$-small colimit of hom-objects $\cd(C,D)$; but these are all $\lambda$-presentable in \cv, hence so too is $GD$.

Suppose conversely that each $GD$ is $\lambda$-presentable. Each $\cd(-,D)$ is $\lambda$-presentable (in fact small-projective) and so each $GD\cdot \cd(-,D)$ is $\lambda$-presentable. But $G$ itself is a $\lambda$-small colimit of the $GD\cdot \cd(-,D)$, since $\cd$ is $\lambda$-small, thus $G$ is also $\lambda$-presentable. 
\endproof

\begin{propo}\label{prop:filtered-finite-decomposition}
Let $\cv$ be a combinatorial monoidal model category satisfying the monoid axiom, $\cm$ a combinatorial model $\cv$-category, and \cd a small $\cv$-category. Then there is a regular cardinal $\lambda$ such that, in $\Int\cm$, any weighted homotopy colimit over \cd is a homotopy $\lambda$-filtered colimit of $\lambda$-small weighted homotopy colimits. 
\end{propo}
\begin{proof}
There is a regular cardinal $\lambda$ such that $[\cd{\op},\cv]$ is $\lambda$-combinatorial with cofibrant replacement preserving $\lambda$-presentable
objects. Consider a weight $G\colon\cd\op\to\cv$ and a diagram $S\colon\cd\to\Int\cm$. Since $G$ is a $\lambda$-filtered colimit of $\lambda$-small weights, the cofibrant replacement $QG$ is a $\lambda$-filtered colimit of $\lambda$-small and cofibrant weights. Since weighted homotopy colimits are functorial
in $\Int\cm$, a homotopy colimit of $S$ weighted by $G$ is a homotopy $\lambda$-filtered colimit of $\lambda$-small weighted homotopy colimits.
\end{proof}

\begin{propo}\label{hpres}
Let \cv be a $\lambda$-combinatorial monoidal model category which is locally $\lambda$-presentable as a closed category and has all objects cofibrant, and 
let \ck be a fibrant \cv-category. Then a $\lambda$-small weighted homotopy colimit of homotopy $\lambda$-presentable objects in \ck is homotopy $\lambda$-presentable.
\end{propo}

\proof
Let \cd be a $\lambda$-small \cv-category, and $H\colon\cd\op\to\cv$ a $\lambda$-small cofibrant weight. 
Let $S\colon \cd\to\ck$ have homotopy $\lambda$-presentable values. We must show that $H*_hS$ is homotopy $\lambda$-presentable; in other words, that $\ck(H*_h S,-)$ preserves homotopy $\lambda$-filtered colimits. 

We have a pointwise weak equivalence 
$$\ck(H*_hS,-) \to [\cd\op,\cv](H,\ck(S,-))$$
but this takes values in $\Int\cv$ and so is in fact a pointwise homotopy equivalence.  Thus, by Lemma~\ref{pres1}, it will suffice to show that 
$[\cd\op,\cv](H,\ck(S,-))$ preserves the homotopy colimit. This functor is the composite 
$$\xymatrix @C3pc {
\ck \ar[r]^-{\ck(S,-)} & [\cd\op,\cv] \ar[rr]^-{[\cd\op,\cv](H,-)} && \Int\cv }$$
which is pointwise homotopy equivalent to the composite
$$\xymatrix @C3pc {
\ck \ar[r]^-{Q\circ \ck(S,-)} & \Int[\cd\op,\cv] \ar[rr]^-{\Int[\cd\op,\cv](H,-)} && \Int\cv. }$$
The first factor $Q\circ\ck(S,-)$ preserves $\lambda$-filtered homotopy colimits by Lemma~\ref{lemma:qe}, while the second factor preserves them by Proposition~\ref{prop:preservation}. \endproof

Recall the following two characterizations of locally $\lambda$-presentable categories. Let \ck be cocomplete and $J\colon\ca\to\ck$ a full subcategory consisting of $\lambda$-presentable objects. Then \ck is locally $\lambda$-presentable if either of the following equivalent conditions hold:
\begin{itemize}
\item each object of \ck is a $\lambda$-filtered colimit of objects in \ca;
\item the induced functor $\ck(J,1)\colon\ck\to[\ca\op,\Set]$ is fully faithful (in other words, $J$ is dense).
\end{itemize}

In our homotopy-theoretic setting we shall consider analogues of both these conditions; it is no longer clear that they are equivalent. 
 
\begin{defi}\label{def6.9}
Let $\cv$ be a combinatorial monoidal model category having all objects cofibrant, and let \ck be a fibrant \cv-category with weighted homotopy colimits. 
Let $J\colon\ca\to\ck$  be a small full subcategory consisting of homotopy $\lambda$-presentable objects. 

We say that \ca exhibits \ck as {\em strongly homotopy locally $\lambda$-presentable} if every object of \ck is a homotopy $\lambda$-filtered colimit of objects in \ca. 

We say that \ca exhibits \ck as {\em homotopy locally $\lambda$-presentable} if the induced \cv-functor
$$\xymatrix{
\ck \ar[r]^-{E} & [\ca\op,\cv] \ar[r]^{Q} & [\ca\op,\cv] }$$
is locally a weak equivalence. 

We say that \ck is {\em strongly homotopy locally $\lambda$-presentable} or {\em homotopy locally $\lambda$-presentable} if there is some such \ca, and we say that \ck is {\em strongly homotopy locally presentable} or {\em homotopy locally presentable} if it is so for some $\lambda$. 
\end{defi}

\begin{propo}\label{prop:strong-weak}
  Every strongly homotopy locally $\lambda$-presentable \cv-category is homotopy locally $\lambda$-presentable.
\end{propo}

\proof 
Let $J\colon\ca\to\ck$ exhibit \ck as strongly homotopy locally $\lambda$-presentable. We must show that 
$$\xymatrix{
\ck \ar[r]^-E & [\ca\op,\cv] \ar[r]^{Q} & [\ca\op,\cv] }$$
is locally a weak equivalence; in other words that 
$$\xymatrix{
\ck(K,K') \ar[r]^-E & [\ca\op,\cv](EK,EK') \ar[r]^-{Q} & [\ca\op,\cv](QEK,QEK') }$$
is a weak equivalence for all $K,K'\in\ck$. 

Since \ck is fibrant, $E$ has fibrant values, and so in particular $EK'$ is fibrant. Thus $q\colon QEK'\to EK'$ is a weak equivalence between fibrant objects, and $QEK$ is cofibrant, thus the map $[\ca\op,\cv](QEK,QEK')\to[\ca\op,\cv](QEK,EK')$ given by composition with $qEK'$ is a weak equivalence. Thus, by naturality of $q$,  it remains to show that the composite
$$\xymatrix{
\ck(K,K') \ar[r]^-{E} & [\ca\op,\cv](EK,EK') \ar[r]^{q^*} & [\ca\op,\cv](QEK,EK') }$$
is a weak equivalence, where $q^*$ denotes the map given by composition with $q\colon QEK\to EK$. 

To do this, let \cl consist of those objects $K$ for which this composite is a weak equivalence for all $K'\in\ck$. First, observe that \cl contains the objects in \ca, for if $K=JA$, then $EK=EJA=\ca(-,A)$ is representable, and so cofibrant, and so $q^*$ is itself a weak equivalence, and we need only check that $E\colon\ck(JA,K')\to[\ca\op,\cv](EJA,EK')$ is one. But this holds by the Yoneda lemma, since $[\ca\op,\cv](EJA,EK')=[\ca\op,\cv](\ca(-,A),EK')\cong (EK')A=\ck(JA,K')$.

Thus if \cl is closed under homotopy $\lambda$-filtered colimits, then it will be all of \ck. Suppose then that $\cj$ is a $\lambda$-filtered ordinary category, and that $\cd=\overline{\cj}$ is the free \cv-category on \cj, and that $K\colon \cd\op\to\cv$ is a cofibrant replacement of $\Delta I\colon \overline{\cj}\op\to\cv$. Let $S\colon\cd\to\ck$ be the \cv-functor induced by a diagram $\cj\to\ck$, and suppose that $S$ takes values in \cl. We must show that the homotopy colimit $K*_h S$ lies in \cl. 

Consider the following diagram
$$\xymatrix{
[\cd\op,\cv](K,\ck(S-,X)) \ar[r] \ar[d]_{[\cd\op,\cv](K,E)} & \ck(K*_hS,X) \ar[d]^{E}\\
[\cd\op,\cv](K,[\ca\op,\cv](ES-,EX)) \ar[d]_{[\cd\op,\cv](K,Q)} % \ar[d]_{[\cd\op,\cv](K,q^*)}  
& [\ca\op,\cv](E(K*_hS),EX) \ar[d]^{Q} \\ %\ar[d]^{q^*}  \\
%[\cd\op,\cv](K,[\ca\op,\cv](QES-,EX)) \ar[r] & [\ca\op,\cv](QE(K*_hS),EX) }$$
[\cd\op,\cv](K,[\ca\op,\cv](QES-,QEX)) \ar[r] & [\ca\op,\cv](QE(K*_hS),QEX) }$$
in which the horizontal arrows are weak equivalences induced by the universal property of the homotopy colimit $K*_hS$ and the fact that $QE$ preserves this homotopy colimit (see Lemma~\ref{lemma:qe}). For each object $D\in\cd$, we know that $SD\in\cl$, and so that the composite 
$$\xymatrix{
\ck(SD,X) \ar[r]^-{E} & [\ca\op,\cv](ESD,EX) \ar[r]^{Q} & [\ca\op,\cv](QESD,QEX) }$$
is a weak equivalence between fibrant objects; since $K$ is cofibrant, it follows 
that the vertical composite on the left of the previous diagram is a weak equivalence, and so that the vertical composite on the right is also a weak equivalence. This proves that $K*_h S$ lies in \cl. % It now follows that \cl is all of \ck, and so that 
% $$\xymatrix{
% \ck(K,K') \ar[r]^-{E} & [\ca\op,\cv](EK,EK') \ar[r]^-{q^*} & [\ca\op,\cv](QEK,EK') }$$
% is a weak equivalence for all $K,K'\in\ck$, and so that $QE\colon\ck\to[\ca\op,\cv]$ is locally a weak equivalence. 
\endproof

\begin{propo}\label{prop:IntMisLP} 
Let $\cv$ be a combinatorial monoidal model category having all objects cofibrant and $\cm$ a combinatorial model $\cv$-category. 
Then $\Int\cm$ is strongly homotopy locally presentable, and so also homotopy locally presentable.
\end{propo}
\begin{proof}
We know that $\Int\cm$ has weighted homotopy colimits. By Remark~\ref{rmk:fibrants-accessible}, there is a regular cardinal $\lambda$ such that each object in $\Int\cm$ is a $\lambda$-filtered colimit of objects from $\Int\cm$ which are $\lambda$-presentable in $\cm$. Moreover, $\cm$ is $\lambda$-combinatorial. 
Thus $\lambda$-filtered colimits are weakly equivalent to homotopy $\lambda$-filtered colimits: see Theorem~\ref{thm:hocolim-colim}. 
Since the same is true in $\cv$, objects from $\Int\cm$ which are $\lambda$-presentable in $\cm$ are homotopy $\lambda$-presentable in $\Int\cm$.
We choose \ca to consist of those objects of $\Int\cm$ which are $\lambda$-presentable in \cm. 

 By Remark~\ref{rmk:fibrants-accessible}, any object of $\Int\cm$ is a $\lambda$-filtered colimit of objects in \ca. Since this $\lambda$-filtered colimit lies in $\Int\cm$, it is a homotopy colimit.
\end{proof}

\begin{theo}\label{th6.12} 
Let $\cv$ be a combinatorial monoidal model category in which all objects are cofibrant, and let $\cm$ be a tractable model $\cv$-category. Then each small homotopy orthogonality class in $\Int\cm$ is strongly homotopy locally presentable, and so also homotopy locally presentable.
\end{theo} 
\begin{proof}
Let $\cm$ be a tractable model  \cv-category and $\ck$ a small homotopy orthogonality class in $\Int\cm$. By Theorem~\ref{thm:orthog-refl},
$\ck$ is homotopy reflective in $\Int\cm$;  let $L$ be a homotopy reflection. Consider a cofibrant weight $G\colon\cd\op\to\cv$ and a diagram $S\colon\cd\to\ck$. By the universal property of the homotopy weighted colimit $G\ast_h JS$, there are weak equivalences 
$$\Int\cm(G*_h JS,M)\to [\cd\op,\cv](G,\Int\cm(JS,M))$$
natural in $M\in\Int\cm$; by the universal property of the homotopy reflection $L(G*_h JS)$ there are weak equivalences 
$$\ck(L(G*_h JS),K)\to\Int\cm(G*_h JS,JK)$$
natural in $K\in\ck$. Taking $M=JK$ in the first  and then composing with the second, we obtain weak equivalences 
$$\ck(L(G*_h JS),K)\to [\cd\op,\cv](G,\Int\cm(JS,JK))\cong [\cd\op,\cv](G,\ck(S,K))$$
natural in $K\in\ck$. This shows that $L(G*_h JS)$ has the universal property of the homotopy colimit $G*_h S$ in \ck, and so that \ck has homotopy weighted colimits. 

By Remark~\ref{rmk:orthog-class-lambda} (1), there is a regular cardinal $\lambda$ such that $\cm$ is $\lambda$-com\-bi\-na\-to\-rial and each object $X$ 
in $\ck$ is a $\lambda$-filtered colimit in $\cm$ of objects from $\ck$ which are $\lambda$-presentable in $\cm$. Hence the same
argument as in the proof of Proposition~\ref{prop:IntMisLP} yields that $\ck$ is strongly homotopy locally $\lambda$-presentable.
\end{proof}

\begin{theo}\label{thm:hlp-we-IntM} 
Let $\cv$ be a combinatorial monoidal model category in which all objects are cofibrant. Then, assuming Vop\v enka's principle, any homotopy locally presentable $\cv$-category $\ck$ admits a weak equivalence $\ck\to\Int\cm$ where $\cm$ is a combinatorial model $\cv$-category. Furthermore, \cm can be taken to be a left Bousfield localization of an (enriched) presheaf category  with respect to a set of morphisms.
\end{theo} 

\begin{proof} Let $J\colon\ca\to\ck$ exhibit \ck as homotopy locally $\lambda$-presentable, where \cv is $\lambda$-combinatorial. We know that the composite
$$\xymatrix{
\ck \ar[r]^-{E} & [\ca\op,\cv] \ar[r]^-{Q} & [\ca\op,\cv] }$$
is locally a weak equivalence, but there is no reason why it should be homotopy surjective on objects. We deal with this via a suitable localization.

For each cofibrant weight $G\in[\ca\op,\cv]$, we have the homotopy colimit $G*_h J$ in \ck, and the corresponding unit 
$$\xymatrix{
G \ar[r]^-{\delta_G} & \ck(J,G*_h J) =  E(G*_h J). }$$
Since the cofibrant replacement map $q\colon QE(G*_h J)\to E(G*_h J)$ is a trivial fibration and $G$ is cofibrant, we may choose a lifting 
$$\xymatrix{
G \ar[r]^-{\gamma_G} & QE(G*_h J) }$$
of $\delta_G$ through $q\colon QE(G*_h J)\to E(G*_h J)$. 
The collection \cf of all such $\gamma_G$ with $G$ cofibrant can be written as the union $\cf=\cup_\lambda\cf_\lambda$, where $\cf_\lambda$ is the (small) set of those $\gamma_G$ for which $G$ is cofibrant and $\lambda$-presentable in $[\ca^{\op},\cv]$. Since all objects of \cv are cofibrant, $[\ca^{\op},\cv]$ is tractable and left proper, and so we may form the enriched left Bousfield localization $\cm_\lambda$ of $[\ca^{\op},\cv]$ with respect to $\cf_\lambda$: see \cite[4.46]{Ba}. Let $\cw_\lambda$ denote the class of weak equivalences in $\cm_\lambda$. 
Then the $\cw_\lambda$ form an increasing chain and we write \cw for its union. By \ref{vopenka}, a left Bousfield localization $\cm$ of $[\ca\op,\cv]$ with respect to $\cf$ exists, has $\cw$ as the class of weak equivalences and is equal to a left Bousfield localization $[\ca\op,\cv]$ with respect to some 
$\cf_\lambda$.  Thus $\cm$ is a combinatorial model $\cv$-category.

The cofibrant objects of \cm are the same as the cofibrant objects of $[\ca\op,\cv]$. The fibrant objects of \cm are those fibrant objects $H$ of $[\ca\op,\cv]$ which are moreover {\em \cf-local}, which here means that they are homotopy orthogonal to the $\gamma_G$. (By the invariance results of Proposition~\ref{prop:hcolimit-uniqueness}, this is independent of the choice of homotopy colimits.) This means that composition with each $\gamma_G$ induces a weak equivalence 
$$[\ca\op,\cv](QE(G*_h J),H)\to [\ca\op,\cv](G,H)$$
in \cv, and so a homotopy equivalence in $\Int\cv$. 

In the following commutative diagram, in which $qEK_*$ denotes composition with $qEK$,
$$\xymatrix{
& \ck(G*_h J,K) \ar[d]_{E} \\
& [\ca\op,\cv](E(G*_h J),EK) \ar[d]_{q^*} \ar[dl]_{Q} \ar `r[r]`[ddr]^{\delta^*_G}[dd]
%\ar@(r,l)[dd]^{\delta^*_G} 
& \\
[\ca\op,\cv](QE(G*_h J),QEK) \ar[r]_{qEK_*} \ar[d]_{\gamma^*_G} & [\ca\op,\cv](QE(G*_h J),EK) \ar[d]_{\gamma^*_G} \\
[\ca\op,\cv](G,QEK) \ar[r]_{qEK_*} & 
[\ca\op,\cv](G,EK) & }$$
the right-hand path $\delta^*_G E$ is a weak equivalence by the universal property of the homotopy colimit $G*_h J$, thus the left-hand path is also a weak equivalence. But the final map $qEK_*$ of the left-hand path is a weak equivalence since $G$ is cofibrant and $qEK$ is a trivial fibration; also, we saw above that the composite $QE$ appearing at the beginning of the left-hand path is a weak equivalence; it follows therefore that the $\gamma^*_G$ appearing in the middle of the left-hand path is a weak equivalence. This proves that $QEK$ is \cf-local, and so that $QE\colon\ck\to[\ca\op,\cv]$ takes values in $\Int\cm$.

Since, as a \cv-category, \cm is just $[\ca\op,\cv]$, the induced \cv-functor $\ck\to\Int\cm$ will still be locally a weak equivalence. It remains to show that it is homotopy surjective; equivalently, that if $H\in\Int[\ca\op,\cv]$ is \cf-local, then it is homotopy equivalent to some $QEK$. But this follows from
the fact that $\gamma_H:H\to QE(H\ast_h J)$ is a homotopy equivalence.
\end{proof}

\begin{propo}\label{prop:hlp-invariance}
Let \cv be a combinatorial monoidal model category in which all objects are cofibrant, and let $W\colon\ck\to\cl$ be a weak equivalence of fibrant \cv-categories. If \cl is homotopy locally $\lambda$-presentable, then so is \ck. 
\end{propo}

\proof
Let $J\colon\cb\to\cl$ exhibit \cl as homotopy locally $\lambda$-presentable. Since $W$ is a weak equivalence, for each $B\in\cb$ there is some $A\in\ck$ for which there is a homotopy equivalence $B\to WA$. If we enlarge \cb so as to include these objects $WA$, the resulting $J\colon\cb\to\cl$ will still exhibit \cl as homotopy locally $\lambda$-presentable. Let $H\colon\ca\to\ck$ be the full subcategory consisting of all such objects $A$, and let $V\colon\ca\to\cb$ be the restriction of $W$. By Proposition~\ref{hococomplete}, we know that \ck is homotopically cocomplete, and that each $A\in\ca$ is homotopically $\lambda$-presentable in \ck.  It remains to prove that $Q\tilde{H}\colon\ck\to\Int[\ca\op,\cv]$ is locally a weak equivalence, where $\tilde{H}$ is given by $\tilde{H}X=\ck(H,X)$. 

Now $V\colon\ca\to\cb$ is a weak equivalence of \cv-categories, and so by Proposition~\ref{prop:DKimpliesQ} the \cv-functor $V^*\colon[\cb\op,\cv]\to[\ca\op,\cv]$ given by restriction along $V$ is part of a Quillen equivalence $V_!\dashv V^*$.  In particular, this means that the functor 
$$QV^*\colon \Int[\cb\op,\cv]\to\Int[\ca\op,\cv]$$
is locally a weak equivalence. Thus the composite 
$$\xymatrix{
\ck \ar[r]^{W} & \cl \ar[r]^-{Q\tilde{J}} & \Int[\cb\op,\cv] \ar[r]^{QV^*} & \Int[\ca\op,\cv] }$$
is locally a weak equivalence, where $\tilde{J}L=\cl(J-,L)$. 

We now have three \cv-functors from \ck to $[\ca\op,\cv]$, namely $Q\tilde{H}$, $V^*\tilde{J}W$, and $QV^*Q\tilde{J}W$; we know that the last is locally a weak equivalence and we want to show that the first is one. 

Now 
$$V^*\tilde{J}WX=V^*\cl(J-,WX)=\cl(JV-,WX)=\cl(WH-,WX)$$
and composition with $W$ induces a morphism 
$$\xymatrix{
\ck(H-,X) \ar[r]^-{W_{H_,X}} & \cl(WH-,WX) }$$
in $[\ca\op,\cv]$ which is a weak equivalence; as $X$ varies in \ck, the $W_{H,X}$  define a \cv-natural 
$$\xymatrix{
\ck \ar@/^1pc/[rrr]^{\tilde{H}}_{~}="1" \ar@/_1pc/[rrr]_{V^*\tilde{J}W}^{~}="2" &&& [\ca\op,\cv] 
\ar@{=>}"1";"2"^{W_{H,-}} }$$
which is a pointwise weak equivalence.

On the other hand $q\colon Q\tilde{J}W\to\tilde{J}W$ is a pointwise trivial fibration, and $V^*$ preserves trivial fibrations and so the composites % in the diagram
$$\xymatrix{
Q\tilde{H} \ar[r]^{q} & \tilde{H} \ar[r]^-{W_{H,-}} & V^*\tilde{J}W & 
QV^*Q\tilde{J}W \ar[r]^{q} & V^*Q\tilde{J}W \ar[r]^{V^*q} & V^*\tilde{J}W }$$
% $$\xymatrix{
% Q\tilde{H} \ar[dd]_q  & QV^*Q\tilde{J}W \ar[d]^{q} \\
% & V^*Q\tilde{J}W \ar[d]^{V^*q} \\
% \tilde{H} \ar[r]^w  & V^*\tilde{J}W }$$
are pointwise weak equivalences. 
% the verticals are pointwise trivial fibrations and the solid horizontal is a pointwise weak equivalence.
Also $Q\tilde{H}$ and $QV^*Q\tilde{J}W$ are pointwise fibrant and cofibrant (so take values in $\Int[\ca\op,\cv]$), while $V^*\tilde{J}W$ is pointwise fibrant. 
Thus by the following lemma, $Q\tilde{H}$ is locally a weak equivalence because $QV^*Q\tilde{J}W$ is one; this completes the proof of the proposition. 

\begin{lemma}
  Let \cm be a model \cv-category and \ck an arbitrary \cv-category. Let $R,S,T\colon\ck\to\cm$ be \cv-functors with fibrant values, and suppose further that $R$ and $T$ have cofibrant values. Let $v\colon R\to S$ and $q\colon T\to S$ be pointwise weak equivalences. Then $R$ is a local weak equivalence if and only if $T$ is one. 
\end{lemma}

\proof
Since the property of being a local weak equivalence can be checked on each hom-object separately, it suffices to consider the case where \ck is small; then we can consider the projective model structure on $[\ck,\cv]$. 

Let $P$ be the (pointwise) cofibrant replacement of the pullback of $v$ and $q$; thus we have a square
$$\xymatrix{
P \ar[r]^{p} \ar[d]_{u} & R \ar[d]^{v} \\ T \ar[r]_{q} & S }$$
of pointwise weak equivalences.  Consider the diagram
$$\xymatrix @C3pc {
\ck(X,Y) \ar[rr]^{R_{X,Y}} \ar[dd]_{T_{X,Y}} \ar[dr]^{S_{X,Y}} &&
\cm(RX,RY) \ar[d]^{\cm(RX,vY)} \\
& \cm(SX,SY) \ar[r]^{\cm(vX,SY)} \ar[d]_{\cm(qX,SY)} & \cm(RX,SY) \ar[d]^{\cm(pX,SY)} \\
\cm(TX,TY) \ar[r]_{\cm(TX,qY)} & \cm(TX,SY) \ar[r]_{\cm(uX,SY)} & \cm(PX,SY) }$$
for objects $X,Y\in\ck$. 
Now $RX$ is cofibrant and $vY$ is a weak equivalence between fibrant objects, so $\cm(RX,vY)$ is a weak equivalence; similarly $\cm(TX,qY)$ is a weak equivalence. On the other hand $SY$ is fibrant and $pX$ is a weak equivalence between cofibrant objects, and so $\cm(pX,SY)$ is a weak equivalence; similarly $\cm(uX,SY)$ is a weak equivalence. It follows that $R_{X,Y}$ is a weak equivalence iff $T_{X,Y}$ is one. 
\endproof

Combining the various results proved so far, we obtain a characterization of homotopy locally presentable categories. 

\begin{theo}\label{thm:hlp-characterization}
Let \cv be a combinatorial monoidal model category in which all objects are cofibrant. Then, assuming Vop\v enka's principle, a fibrant \cv-category \ck is homotopy locally $\lambda$-presentable if and only if there is a weak equivalence $\ck\to\Int\cm$ for some  combinatorial model category \cm. Furthermore \cm can be taken to be a left Bousfield localization of an (enriched) presheaf category  with respect to a set of morphisms. 
\end{theo}

\proof
 If there is such a weak equivalence $\ck\to\Int\cm$ then \ck is homotopy locally presentable by Propositions~\ref{prop:IntMisLP} and ~\ref{prop:hlp-invariance}; this does not require Vop\v enka's principle. We do use it for the converse, which is Theorem~\ref{thm:hlp-we-IntM}.
\endproof 

\begin{rem}\label{re6.14}
(1) The proof of Theorem~\ref{thm:hlp-we-IntM} shows that a homotopy locally presentable $\cv$-category admits a weak equivalence into a category of models of a $\lambda$-small weighted homotopy limit sketch.

(2) Using \cite{V}, we can replace the existence of weighted homotopy colimits in the definition of homotopy locally presentable $\cv$-category by the existence of homotopy colimits and homotopy copowers.

(3) We have generalized results given in \cite{R1} from $\SSet$ to any monoidal model category $\cv$ having all objects cofibrant. The paper \cite{R1} contained a stronger formulation of \ref{thm:hlp-we-IntM} which asserted that each fibrant simplicial category weakly equivalent to $\Int\cm$, where $\cm$ is a combinatorial simplicial model category, is homotopy locally presentable. This stronger formulation was withdrawn in \cite{R2}. % {\grey We have not been able to generalize this stronger formulation.}
In fact Vop\v enka's principle also seems to be needed for the arguments in \cite{R1}. 
\end{rem}

On the other hand, the following result does not require Vop\v enka's principle: 

\begin{propo}\label{prop:shlp-invarianece}
Let $\cv$ be a combinatorial monoidal model category in which all objects are cofibrant. Let $\ck$ and $\cl$ be fibrant $\cv$-categories, let 
$W_1:\ck\to\cl$ and $W_2:\cl\to\ck$ be weak equivalences, and suppose that $W_2W_1 X$ is homotopy equivalent to $X$ for each $X\in\ck$, and that $W_1 W_2 Y$ is homotopy equivalent to $Y$ for each $Y\in\cl$. 
Then $\ck$ is strongly homotopy locally $\lambda$-presentable if and only if $\cl$ is strongly homotopy locally $\lambda$-presentable.
\end{propo}

\proof
Clearly it suffices to prove that if \cl is strongly homotopy locally $\lambda$-presentable then so is \ck. By Proposition~\ref{hococomplete} we know that \ck is homotopically cocomplete. Let $J\colon\cb\to\cl$ exhibit \cl as strongly homotopically locally $\lambda$-presentable. Let $H\colon\ca\to\ck$ be the full image of \cb under $W_2\colon\cl\to\ck$, and let $V\colon \cb\to\ca$ be the restriction of $W_2$. 

Each $A\in\ca$ has the form $W_2 B$ for some $B\in\cb$. Now $B$ is homotopy $\lambda$-presentable, and $W_1A=W_1W_2B$ is homotopically equivalent to $B$ so is also homotopy $\lambda$-presentable.
It follows by Proposition~\ref{hococomplete} that $A$ is homotopy $\lambda$-presentable in \ck.

Finally, any $X\in\ck$ is homotopically equivalent to $W_2 W_1 X$, and we can write $W_1 X=\hocolim_i JB_i$ as a homotopy $\lambda$-filtered colimit of objects in \cb, and $W_2$ preserves homotopy colimits, so 
$$X\simeq W_2 W_1 X=W_2\hocolim_i JB_i \simeq \hocolim_i W_2JB_i$$
with each $W_2 JB_i\in\ca$, and so $X$ is a homotopy $\lambda$-filtered colimit of objects in \ca. 
\endproof

The following result was proved by D. Dugger \cite{D} in the case of simplicial model categories.

\begin{theo}\label{thm:GuillouMay-comparison}
Let \cv be a combinatorial monoidal model category in which all objects are cofibrant, and let \cn be a combinatorial model \cv-category. Then,  
there is a Quillen equivalence $U\colon\cn\to\cm$ where \cm is a left Bousfield localization of a \cv-presheaf category with respect to a set of morphisms;
\end{theo}

\proof  
Define $\ca$ as in the proof of Theorem~\ref{thm:hlp-we-IntM}, with $J\colon\ca\to\cn$ the inclusion; this takes values in $\Int\cn$. Since \cn is cocomplete, the induced functor $\tilde{J}\colon\cn\to[\ca\op,\cv]$, sending $N\in\cn$ to $\cn(J-,N)\colon\ca\op\to\cv$, has a left adjoint $L$ sending $M\in[\ca\op,\cv]$ to $M*J$. Since $J$ has cofibrant values, $\tilde{J}$ preserves fibrations and trivial fibrations, and therefore $L\dashv \tilde{J}$ is a Quillen adjunction. We know that the induced map $\Int\cn\to\Int[\ca\op,\cv]$ is locally a weak equivalence, so that the derived functor $\Ho\cn\to\Ho[\ca\op,\cv]$ is fully faithful, and the counit of the derived adjunction is invertible. This in turn implies that the map 
$$\xymatrix{
LQ\tilde{J}N \ar[r]^{Lq} & L\tilde{J}N \ar[r]^{e} & N }$$
is a weak equivalence for all fibrant $N$, where $e$ is the counit of the adjunction $L\dashv \tilde{J}$. 

For all $G\in[\ca\op,\cv]$, we have the  composite
$$\xymatrix{
G \ar[r]^-{n} & \tilde{J}LG \ar[r]^-{\tilde{J}r} & \tilde{J}RLG }$$
where $n$ is the unit of the adjunction; this defines a natural map $s\colon 1\to \tilde{J}RL$. 
As in \cite[3.2]{D} or \cite{R}, choose 
$\lambda$ so that: 
\begin{itemize}
\item $Q$ and $R$ preserve $\lambda$-filtered colimits
%\item weak equivalences in \cv closed under $\lambda$-filtered limits 
\item each $A\in\ca$ is $\lambda$-presentable in \cn
\item $Q$ preserves $\lambda$-presentability
\end{itemize}
and let \cm be the localization of $[\ca\op,\cv]$ with respect to the set \cf of all $Qs_G\colon QG\to Q\tilde{J}RLG$ with $G$ cofibrant and $\lambda$-presentable. 

By virtue of the Quillen adjunction, for any  $G\in\Int[\ca\op\cv]$ and any $N\in\Int\cn$, the composite 
$$\xymatrix{
\cn(RLG,N) \ar[r]^-{Q\tilde{J}} & 
[\ca\op,\cv](Q\tilde{J}RLG,Q\tilde{J}N) \ar[r]^-{(Qs_G)^*} & 
[\ca\op,\cv](QG,Q\tilde{J}N) }$$
is a weak equivalence. Also the first map $Q\tilde{J}$ is a weak equivalence because $\Int\cn$ is homotopy locally presentable, and so the second map $(Qs_G)^*$ is a weak equivalence in $\Int\cv$, and so a homotopy equivalence. Thus $Q\tilde{J}N$ is \cf-local, and so by the universal property of the localization \cm, the Quillen adjunction $L\dashv U$ factorizes through \cm. The counit of the derived adjunction is still invertible, but we should check that the unit of the derived adjunction is so too; in other words, that for each cofibrant $G$, the map $s_G\colon G\to \tilde{J}RLG$ is a weak equivalence in \cm; or, equivalently, 
that $Qs_G\colon QG\to Q\tilde{J}RLG$ is a weak equivalence. But $G$ is a $\lambda$-filtered colimit of $\lambda$-presentable objects, while both $Q$ and $Q\tilde{J}RL$ preserve $\lambda$-filtered colimits and the weak equivalences are closed under $\lambda$-filtered colimits, so it will suffice to show that $Qs_G$ is a weak equivalence for all $\lambda$-presentable objects $G$. This is true precisely because we have localized with respect to all such $Qs_G$. 
\endproof

% that each \cf-local $H$ is homotopy orthogonal to $Qs_G\colon QG\to Q\tilde{J}RLG$. 

% This will be the case if and only if the map 
% $$\xymatrix{
% [\ca\op,\cv](Q\tilde{J}RLG,H) \ar[r]^-{(Qs_G)^*} & [\ca\op,\cv](QG,H) }$$
% is a weak equivalence. Since $H$ is \cf-local, $(Qs_G)^*$ will be a weak equivalence whenever $G$ is $\lambda$-presentable; but any cofibrant $G$ is a $\lambda$-filtered colimit of cofibrant and $\lambda$-presentable objects.
% Now $Q$ and $Q\tilde{J}RL$ preserve $\lambda$-filtered colimits, thus $(Qs_G)^*$ is a $\lambda$-filtered limit of weak equivalences, and so a weak equivalence. 
%\endproof

\appendix

\section{Enriched cofibrant replacement}\label{app}

In this appendix we discuss issues related to the existence of an enriched cofibrant replacement functor. The bad news is that this is rather rare: we show below that if such an enriched cofibrant replacement functor exists for a monoidal model category \cv with cofibrant unit, then all objects \cv must be cofibrant.  The good news is, as explained in \cite[Proposition~24.2]{S}, that if all objects of \cv are cofibrant, then any cofibrantly generated model \cv-category does have an enriched cofibrant replacement functor, formed by a straightforward adaptation of the small object argument to the enriched setting. 

\begin{propo}\label{prop:all-cofibrant}
  Suppose that \cv is a monoidal model category with cofibrant unit, and that the \cv-natural transformation $q\colon Q\to 1$ exhibits $Q\colon\cv\to\cv$ as a cofibrant replacement \cv-functor. Then all objects of \cv are cofibrant.
\end{propo}

\proof
Since $I$ is cofibrant, there exists a section $s\colon I\to QI$ to $q\colon QI\to I$.
Let $X$ be an arbitrary object. Since $Q$ is a \cv-functor, it acts on the internal hom $[I,X]$ as a map $Q\colon [I,X]\to[QI,QX]$. Now in the diagram
$$\xymatrix{
[I,X] \ar[r]^-{Q} \ar[dr]_{[q,1]} & [QI,QX] \ar[r]^{[s,QX]} \ar[d]^{[QI,q]} & [I,QX] \ar[d]^{[I,q]} \\
& [QI,X] \ar[r]_{[s,X]} & [I,X] }$$
the triangular region commutes by naturality of $q$, and the rectangular region commutes by associativity of composition. The lower composite $[s,X][q,X]$ is $[qs,X]$ which is the identity. Thus the vertical map $[I,q]\colon[I,QX]\to[I,X]$ has a section; but up to isomorphism this is just $q\colon QX\to X$. This proves that $X$ is cofibrant.
\endproof

\bibliographystyle{plain}
% \bibliography{rosicky}
% \bibliography{my}

\begin{thebibliography}{10}

\bibitem{AR}
Ji{\v{r}}{\'{\i}} Ad{\'a}mek and Ji{\v{r}}\'\i\xspace Rosick{\'y}.
\newblock {\em Locally presentable and accessible categories}, volume 189 of
  {\em London Mathematical Society Lecture Note Series}.
\newblock Cambridge University Press, Cambridge, 1994.

\bibitem{Albert-Kelly}
M.~H. Albert and G.~M. Kelly.
\newblock The closure of a class of colimits.
\newblock {\em J. Pure Appl. Algebra}, 51(1-2):1--17, 1988.

\bibitem{Ba}
Clark Barwick.
\newblock On left and right model categories and left and right {B}ousfield
  localizations.
\newblock {\em Homology, Homotopy Appl.}, 12(2):245--320, 2010.

\bibitem{Bk} 
Tibor Beke.
\newblock Sheafifiable homootpy model categories.
\newblock {\em Math. Proc. Cambr. Phil. Soc.} 129:447--475, 2000.

\bibitem{BM}
Clemens Berger and Ieke Moerdijk.
\newblock On the homotopy theory of enriched categories.
\newblock {\em Quart. J. Math.}, 64:805--846, 2013.

\bibitem{Be}
Julia~E. Bergner.
\newblock A model category structure on the category of simplicial categories.
\newblock {\em Trans. Amer. Math. Soc.}, 359(5):2043--2058, 2007.

\bibitem{Be1}
Julia~E. Bergner.
\newblock A survey of {$(\infty,1)$}-categories.
\newblock In {\em Towards higher categories}, volume 152 of {\em IMA Vol. Math.
  Appl.}, pages 69--83. Springer, New York, 2010.

\bibitem{BKPS}
G.~J. Bird, G.~M. Kelly, A.~J. Power, and R.~H. Street.
\newblock Flexible limits for {$2$}-categories.
\newblock {\em J. Pure Appl. Algebra}, 61(1):1--27, 1989.

\bibitem{Bo} 
John Bourke.
\newblock A colimit decomposition for homotopy algebras in {C}at.
\newblock {\em Appl. Categ. Structures}, 22(1):13--28, 2014. 

\bibitem{C}
Denis-Charles Cisinski.
\newblock Th\'eories homotopiques dans les topos.
\newblock {\em J. Pure Appl. Algebra}, 174(1):43--82, 2002.

\bibitem{D}
Daniel Dugger.
\newblock Combinatorial model categories have presentations.
\newblock {\em Adv. Math.}, 164(1):177--201, 2001.

\bibitem{GuillouMay-enrichedhomotopy} 
\newblock Bertrand Guillou and J.P. May.
\newblock {\em Enriched model categories and presheaf categories}, 
\newblock preprint, arXiv:1110.3567, 2011.


\bibitem{H}
Philip~S. Hirschhorn.
\newblock {\em Model categories and their localizations}, volume~99 of {\em
  Mathematical Surveys and Monographs}.
\newblock American Mathematical Society, Providence, RI, 2003.

\bibitem{Ho}
Mark Hovey.
\newblock {\em Model categories}, volume~63 of {\em Mathematical Surveys and
  Monographs}.
\newblock American Mathematical Society, Providence, RI, 1999.

\bibitem{K}
G.~M. Kelly.
\newblock Structures defined by finite limits in the enriched context. {I}.
\newblock {\em Cahiers Topologie G\'eom. Diff\'erentielle}, 23(1):3--42, 1982.

% \bibitem{K1}
% G.~M. Kelly.
% \newblock Basic concepts of enriched category theory.
% \newblock {\em Repr. Theory Appl. Categ.}, (10):vi+137 pp. (electronic), 2005.
% \newblock Originally published as LMS Lecture Notes 64, 1982.

% \bibitem{Kelly-amiens}
% G.~M. Kelly.
% \newblock Structures defined by finite limits in the enriched context. {I}.
% \newblock {\em Cahiers Topologie G\'eom. Diff\'erentielle}, 23(1):3--42, 1982.

\bibitem{vcat}
G.~M. Kelly and Stephen Lack.
\newblock \cv-{C}at is locally presentable or locally bounded if \cv  is so.
\newblock {\em Theory Appl. Categ.}, 8:555--575, 2001.

\bibitem{La2}
Stephen Lack.
\newblock A {Q}uillen model structure for 2-categories.
\newblock {\em $K$-Theory}, 26(2):171--205, 2002.

\bibitem{La}
Stephen Lack.
\newblock Homotopy-theoretic aspects of 2-monads.
\newblock {\em J. Homotopy Relat. Struct.}, 2(2):229--260, 2007.

\bibitem{La1}
Stephen Lack.
\newblock A 2-categories companion.
\newblock In {\em Towards higher categories}, volume 152 of {\em IMA Vol. Math.
  Appl.}, pages 105--191. Springer, New York, 2010.

\bibitem{La3}
Stephen Lack.
\newblock A {Q}uillen model structure for {G}ray-categories.
\newblock {\em Journal of K-theory}, 8(2):183--221, 2011.

\bibitem{LR1}
Stephen Lack and Ji{\v{r}}{\'{\i}} Rosick{\'y}.
\newblock Enriched weakness.
\newblock {\em J. Pure Appl. Algebra}, 216(8-9):1807--1822, 2012.

\bibitem{L}
Jacob Lurie.
\newblock {\em Higher topos theory}, volume 170 of {\em Annals of Mathematics
  Studies}.
\newblock Princeton University Press, Princeton, NJ, 2009.

\bibitem{MP}
Michael Makkai and Robert Par{\'e}.
\newblock {\em Accessible categories: the foundations of categorical model
  theory}, volume 104 of {\em Contemporary Mathematics}.
\newblock American Mathematical Society, Providence, RI, 1989.

\bibitem{M}
F~Muro.
\newblock Dwyer-Kan homotopy theory of enriched categories.
\newblock arXiv:1201.1575, 2012.

\bibitem{R1}
J.~Rosick{\'y}.
\newblock On homotopy varieties.
\newblock {\em Adv. Math.}, 214(2):525--550, 2007.

\bibitem{R2}
J.~Rosick{\'y}.
\newblock Corrigendun to ``On homotopy varieties".
\newblock {\em Adv. Math.}, 259:841--842, 2014.


\bibitem{R}
Ji{\v{r}}{\'{\i}} Rosick{\'y}.
\newblock Generalized {B}rown representability in homotopy categories.
\newblock {\em Theory Appl. Categ.}, 14:no. 19, 451--479, 2005.

\bibitem{RT}
J. Rosick{\'y} and W. Tholen.
\newblock Left-determined model categories and universal homotopy theories.
\newblock{\em Trans. Amer. math. Soc.}, 355:no. 9, 3611--3623, 2003.

\bibitem{SS}
Stefan Schwede and Brooke~E. Shipley.
\newblock Algebras and modules in monoidal model categories.
\newblock {\em Proc. London Math. Soc. (3)}, 80(2):491--511, 2000.

\bibitem{S}
Michael Shulman.
\newblock Homotopy limits and colimits and enriched homotopy theory.
\newblock arXiv:math/0610194v3, 2009.

\bibitem{V}
L.~Vok\v r\'\i{}nek.
\newblock Homotopy weighted colimits.
\newblock arXiv:1201.2970, 2012.


\end{thebibliography}

\end{document}